\documentclass{amsart}
\usepackage{amssymb,amsfonts, color,epsf}
\usepackage [cmtip,arrow]{xy}
\xyoption{all}
\usepackage {pb-diagram,pb-xy}
\usepackage{enumerate}
\usepackage{accents}
\usepackage[noautoscale]{youngtab}
\usepackage{subfigure}
\usepackage[utf8]{inputenc}

\ifx\pdfpageheight\undefined
\usepackage[dvips,colorlinks=true,linkcolor=blue,citecolor=red,%
urlcolor=green]{hyperref}
\usepackage[dvips]{graphicx}
\makeatletter
\edef\Gin@extensions{\Gin@extensions,.mps}
\makeatother
\else
\usepackage[pdftex]{graphicx}
\usepackage[bookmarksopen=false,pdftex=true,breaklinks=true,%
backref=page,pagebackref=true,plainpages=false,%
hyperindex=true,pdfstartview=FitH,colorlinks=true,%
pdfpagelabels=true,colorlinks=true,linkcolor=blue,%
citecolor=red,urlcolor=green,hypertexnames=false%
]%
{hyperref}
\fi
\usepackage{bm}

\usepackage{tikz-cd}
\makeatletter
\tikzset{
	column sep/.code=\def\pgfmatrixcolumnsep{\pgf@matrix@xscale*(#1)},
	row sep/.code   =\def\pgfmatrixrowsep{\pgf@matrix@yscale*(#1)},
	matrix xscale/.code=%
	\pgfmathsetmacro\pgf@matrix@xscale{\pgf@matrix@xscale*(#1)},
	matrix yscale/.code=%
	\pgfmathsetmacro\pgf@matrix@yscale{\pgf@matrix@yscale*(#1)},
	matrix scale/.style={/tikz/matrix xscale={#1},/tikz/matrix yscale={#1}}}
\def\pgf@matrix@xscale{1}
\def\pgf@matrix@yscale{1}
\makeatother

\usepackage{amsmath,amssymb,amsfonts}
\newtheorem{theorem}{Theorem}
\newtheorem{lemma}{Lemma}[section]
\newtheorem{corollary}{Corollary}
\newtheorem{proposition}{Proposition}[section]

\newtheorem*{claim*}{Claim}

\newtheorem*{theorem*}{Theorem}
\newtheorem*{corollary*}{Corollary}
\theoremstyle{definition}
\newtheorem{definition}{Definition}[section]
\newtheorem{example}{Example}

\newtheorem{notation}{Notation}

\usepackage{algorithm}
\usepackage{algpseudocode}

\usepackage{tikz}

\algnewcommand\algorithmicinput{\textbf{Input:}}
\algnewcommand\INPUT{\item[\algorithmicinput]}
\algnewcommand\algorithmicoutput{\textbf{Output:}}
\algnewcommand\OUTPUT{\item[\algorithmicoutput]}

\algnewcommand\algorithmicproc{\textbf{Procedure:}}
\algnewcommand\PROCEDURE{\item[\algorithmicproc]}
\algnewcommand\algorithmiccomplexity{\textbf{Complexity:}}
\algnewcommand\COMPLEXITY{\item[\algorithmiccomplexity]}

\newlength{\continueindent}
\usepackage{etoolbox}
\makeatletter
\newcommand*{\ALG@customparshape}{\parshape 2 \leftmargin \linewidth \dimexpr\ALG@tlm+\continueindent\relax \dimexpr\linewidth+\leftmargin-\ALG@tlm-\continueindent\relax}
\apptocmd{\ALG@beginblock}{\ALG@customparshape}{}{\errmessage{failed to patch}}
\makeatother

\theoremstyle{remark}
\newtheorem{remark}{Remark}

\definecolor{DarkBlue}{rgb}{0,0.1,0.55}

\numberwithin{equation}{section}

\newcommand{\abs}[1]{\lvert#1\rvert}

\newcommand {\hide}[1]{}
\newcommand {\sign} {\mbox{\bf sign}}
\newcommand{\zero}{\mbox{\bf zero}}

\newcommand {\junk}[1]{}

\newcommand {\R} {\mathrm{R}}

\newcommand {\ZZ} {{\rm Z}}
\newcommand {\RR} {{\mathcal R}}

\newcommand {\Der} {{\rm Der}}

\newcommand {\eps} {{\varepsilon}}

\newcommand {\dist} {{\rm dist}}

\newcommand{\card}{\mathrm{card}}

\def\addots{\mathinner{\mkern1mu
		\raise1pt\vbox{\kern7pt\hbox{.}}
		\mkern2mu\raise4pt\hbox{.}\mkern2mu
		\raise7pt\hbox{.}\mkern1mu}}

\DeclareMathOperator{\distance} {dist}
\DeclareMathOperator{\proj} {Proj}

\newcommand{\Cc}{\mathrm{Cc}}

\newcommand{\interior}{\mathrm{int}}

\newcommand{\BElim}  {\mathrm{BElim}}
\newcommand{\SIGN}  {\mathrm{SIGN}}

	\title[Improved effective {\L}ojasiewicz inequality and applications]
	{
		Improved effective {\L}ojasiewicz inequality and applications
	}
	\author{Saugata Basu}
	\address{Department of Mathematics,
		Purdue University, West Lafayette, IN 47906, U.S.A.}
	\email{sbasu@math.purdue.edu}
	
	\author{Ali Mohammad-Nezhad}
	\address{Department of Mathematical Sciences, 
		Carnegie Mellon University, Pittsburgh, PA 15123, U.S.A.}
	\email{anezhad@andrew.cmu.edu}

\begin{document}

\begin{abstract}
Let $\R$ be a real closed field.
Given a closed and bounded semi-algebraic set $A \subset \R^n$ 
and semi-algebraic continuous functions $f,g:A \rightarrow \R$, 
such that $f^{-1}(0) \subset g^{-1}(0)$, there exist $N$ and $c \in \R$, such that
the inequality ({\L}ojasiewicz inequality) $|g(x)|^N \le c \cdot |f(x)|$ holds for all $x \in A$.
In this paper we consider the case when $A$ is defined by a quantifier-free formula with atoms
of the form $P = 0, P >0, P \in \mathcal{P}$ for some finite subset of polynomials 
$\mathcal{P} \subset \R[X_1,\ldots,X_n]_{\leq d}$, and the graphs of $f,g$ are also 
defined by quantifier-free formulas with atoms
of the form $Q = 0, Q >0, Q \in \mathcal{Q}$, for some finite set
$\mathcal{Q} \subset \R[X_1,\ldots,X_n,Y]_{\leq d}$.
We prove that the {\L}ojasiewicz exponent $N$ in this case is bounded by 
$(8 d)^{2(n+7)}$.
Our bound depends on $d$ and $n$, but is independent of the combinatorial parameters, namely the cardinalities of 
$\mathcal{P}$ and $\mathcal{Q}$.
The previous best known upper bound in this generality appeared in \textit{P. Solern\'o, Effective {\L}ojasiewicz Inequalities in Semi-algebraic Geometry, Applicable Algebra in Engineering, Communication and Computing (1991)} and depended on the sum of degrees of the polynomials defining $A,f,g$ and thus implicitly on the cardinalities of $\mathcal{P}$ and 
$\mathcal{Q}$.
As a consequence
we improve the current best error bounds for polynomial systems under some conditions.
Finally, 
as an abstraction of the notion of independence of the {\L}ojasiewicz exponent from the combinatorial parameters occurring in the descriptions of the given pair of functions, 
we prove a version of {\L}ojasiewicz inequality in polynomially bounded o-minimal structures. We prove the existence of a common {\L}ojasiewicz exponent for certain combinatorially defined infinite (but not necessarily definable) families of pairs of functions. This improves a prior result of Chris Miller
(\textit{C. Miller, Expansions of the real field with power functions, Ann. Pure Appl. Logic 
(1994)}).
\end{abstract}
\maketitle
\tableofcontents

\section{Introduction}
\label{sec:intro}
L. Schwartz conjectured that if $f$ is a real analytic function and $T$ a distribution in some open subset
$\Omega \subset \R^n$, then there exists a distribution $S$ satisfying $f S = T$.
As a main tool in proving this conjecture, {\L}ojasiewicz \cite{Loj4} proved that if $V$ is the set of real zeros of $f$, 
and $x$ in a sufficiently small neighborhood of a point $x_0$ in $V$, 
there exists a constant $d$ 
such that
\[
|f(x)| \geq d \cdot \dist(x,V)^d,
\]
where $\dist(x,V)$ denotes the distance of $x$ from $V$. 
In case $f$ is a polynomial the result was obtained by H\"ormander \cite{Hormander2}.

Several variants of {\L}ojasiewicz inequality have appeared in the literature both in the semi-algebraic and analytic categories.
In the semi-algebraic category,
the following slightly more general version of the above inequality appears in \cite[Corollary~2.6.7]{BCR98}.

Unless otherwise specified $\R$ is a fixed real closed field for the rest of the paper.

Let $A \subset \R^n$ be a closed and bounded 
semi-algebraic set and let $f,g: A \to \R$ be continuous semi-algebraic functions. 
Furthermore, suppose that $f^{-1}(0) \subset g^{-1}(0)$. Then there exist~
$c \in \R$ and a rational number $\rho$ depending on $A$, $f$, and $g$, such that 
\begin{equation}
\label{Orig_Lojasiewicz}
    |g(x)|^{\rho} \le c\cdot |f(x)|, \quad \forall x \in A.
\end{equation}
We denote the infimum of $\rho$ by $\mathcal{L}(f,g \mid A)$ which is called the \emph{{\L}ojasiewicz exponent}. 

The inequality \eqref{Orig_Lojasiewicz} is usually called the \emph{{\L}ojasiewicz inequality} 
and has found many applications (independent of the division problem of L. Schwartz)  --
for example, in singularity theory, partial differential equations, and optimization.
We survey some of these applications later in the paper 
and improve some of these results using the version of {\L}ojasiewicz inequality proved in the current paper.

Driven by the applications mentioned above there has been a lot of interest in obtaining effective bounds on 
$\mathcal{L}(f,g \mid A)$.

\section{Main results}
\label{sec:results}
In this paper we prove new quantitative versions of the inequality \eqref{Orig_Lojasiewicz} in the semi-algebraic 
(and more generally in the o-minimal context).
Before stating our results we introduce a few necessary definitions.

\begin{definition}[$\mathcal{P}$-formulas and semi-algebraic sets]
\label{def:P-formula}
Let $\mathcal{P} \subset \R[X_1,\ldots,X_n]$, $\mathcal{Q} \subset \R[X_1,\ldots,X_n,Y]$ be a finite sets of polynomials. We will call a quantifier-free 
first-order formula (in the theory of the reals)  with atoms $P=0, P >0, P< 0, P \in \mathcal{P}$ to be a \emph{$\mathcal{P}$-formula}. 
Given any first-order formula $\Phi(X_1,\ldots,X_n)$  
in the theory of the reals (possibly with quantifiers), we will 
denote by 
$\RR(\Phi,\R^n)$ the set of points of $\R^n$ satisfying $\Phi$, and call $\RR(\Phi,\R^n)$ the \emph{realization of 
$\Phi$}.
We will call the realization of a $\mathcal{P}$-formula 
a 
\emph{$\mathcal{P}$-semi-algebraic set}.
A \emph{$\mathcal{Q}$-semi-algebraic function} is a function whose graph is a $\mathcal{Q}$-semi-algebraic set.
\end{definition}

We denote by $\R[X_1,\ldots,X_n]_{\leq d}$ the subset of polynomials in $\R[X_1,\ldots,X_n]$ with
degrees $\leq d$.

\subsection{Semi-algebraic case}
We prove the following theorem in the semi-algebraic setting which improves the currently the best known upper bound \cite{S91} in a significant way (see Section~\ref{sec:comb-alg} below).

\begin{theorem}
\label{thm:Lojasiewicz}
Let $d \geq 2$, $\mathcal{P} \subset \R[X_1,\ldots,X_n]_{\leq d},
\mathcal{Q} \subset \R[X_1,\ldots,X_n,Y]_{\leq d}$,
$A \subset \R^n$ a closed and bounded $\mathcal{P}$-semi-algebraic set,
and $f,g: A \rightarrow \R$ continuous $\mathcal{Q}$-semi-algebraic functions, 
satisfying $f^{-1}(0) \subset g^{-1}(0)$.

Then there exist $c = c(A,f,g) \in \R$ and 
$N \leq (8 d)^{2(n+7)}$
such that for all $x \in A$,
\begin{equation}
\label{eqn:thm:Lojasiewicz:0}
|g(x)|^{N}  \leq c \cdot |f(x)|.
\end{equation}
In other words,
\begin{equation}
\label{eqn:thm:Lojasiewicz}
\mathcal{L}(f,g \mid A) \leq (8 d)^{2(n+7)} = 
d^{O(n)}.
\end{equation}

In the special case where $\R = \mathbb{R}$ and  
\[
\mathcal{P} \subset \mathbb{Z}[X_1,\ldots,X_n]_{\leq d},
\mathcal{Q} \subset \mathbb{Z}[X_1,\ldots,X_n,Y]_{\leq d},
\]
and the bit-sizes of the coefficients of the polynomials in 
$\mathcal{P},\mathcal{Q}$ are bounded by $\tau$, 
there exists 
\begin{equation}
\label{eqn:thm:Lojasiewicz:1}
c \leq \min\{2^{\tau d^{O(n^2)}},2^{\tau d^{O(n\log d)}}\}
\end{equation}
such that the inequality~\eqref{eqn:thm:Lojasiewicz:0} holds with $N =(8 d)^{2(n+7)}$.
\end{theorem}

\begin{remark}[Sharpness]
\label{rem:thm:Lojasiewicz:1}
The inequality \eqref{eqn:thm:Lojasiewicz} is nearly tight. The following slight modification
of examples given in ~\cite[Page~2]{S91} or~\cite[Example~15]{JKS92} shows that 
right hand side of inequality \eqref{eqn:thm:Lojasiewicz} cannot be made smaller than $d^n$.
The constants  ($8$ in the base and $14$ in the exponent) in our bound can possibly be improved (for example 
by using a better estimation in the inequality~\eqref{eqn:prop:BElim}
in Proposition~\ref{prop:BElim} and using a slightly more accurate degree bound). 
However, this would lead to a much more unwieldy statement which we prefer to avoid.
The coefficient $2$ of $n$ in the exponent however seems inherent to our method.

\begin{example}\label{ex:exponential_dependence}
Let $A:=\{x \in \mathbb{R}^n \mid x_1^2+\cdots+x_n^2 \le 1\}$ be the compact semi-algebraic set and consider the semi-algebraic functions $f,g:A \to \mathbb{R}$ defined by
\begin{align*}
 f&:=|X_2-X_1^{d_1}|+ \cdots + |X_n - X_{n-1}^{d_{n-1}}| + |X_n^{d_n}|,\\
 g&:=\sqrt{X_1^2+\cdots+X_n^2}.
\end{align*}
It is easy to see that $f^{-1}(0)=g^{-1}(0) = \{0\}$. Then for sufficiently small $|t|$  the vector $x(t):=(t,t^{d_1},\ldots,t^{d_1\cdots d_{n-1}})$ belongs to $A$ and we have
\begin{align*}
    f(x(t))=|t|^{d_1\cdots d_n}, \qquad g(x(t))=\sqrt{t^2+\cdots+t^{2d_1\cdots d_{n-1}}},  
\end{align*}
which implies $|g(x(t))|^{d_1\cdots d_n} \le c \cdot |f(x(t))|$ for some positive constant $c$. Letting $d_1=\cdots=d_n=d$, then it follows that $\mathcal{L}(f,g \mid A) \ge d^n$.
\end{example}
\end{remark}

\begin{remark}[Independence from the combinatorial parameter]
\label{rem:thm:Lojasiewicz:2}
An important feature of the bound in Theorem~\ref{thm:Lojasiewicz}
is that the right hand side of inequality ~\eqref{eqn:thm:Lojasiewicz} 
depends only on the maximum degree of the polynomials 
in $\mathcal{P} \cup \mathcal{Q}$ and is \emph{independent of the cardinalities of the sets 
$\mathcal{P},\mathcal{Q}$}. This is not the case for the previous best known general bound due to Solern\'o 
\cite{S91} which depended on the \emph{sum} of the degrees of the polynomials appearing in the descriptions of
$A$, $f$ and $g$ and thus implicitly on the number of polynomials involved in these descriptions.
This fact,  that our bound is independent of the number of polynomials,  plays an important role in the applications
that we discuss later in the paper. For example, it is exploited crucially in the proof of Theorem~\ref{thm:error_bound_semialg_set} (see below).
Also note that the feature of being independent of combinatorial parameters is also present in some prior work that we discuss in detail in Section~\ref{subsec:prior-sa}. But these results (notably that of Koll\'ar \cite{K99} and also
\cite{OSS21b}) come with certain important restrictions and/or with a worse bound. In contrast, our result is completely general and nearly optimal.
\end{remark}

\begin{remark}[Separation of combinatorial and algebraic parts]
\label{rem:thm:Lojasiewicz:3}
Separating the roles of combinatorial and algebraic
parameters has a long history in quantitative real algebraic geometry.  
We include (see Section~\ref{sec:comb-alg} below) a discussion and several prior examples of such results.
The {\L}ojasiewicz inequality is clearly a foundational result in real algebraic geometry. Hence, asking for a similar distinction in quantitative bounds on the {\L}ojasiewicz exponent
is a very natural question. Finally, the underlying idea behind making this distinction
allows us to formulate and prove a version of the
{\L}ojasiewicz inequality valid over polynomially bounded o-minimal structures (see Theorem~\ref{thm:uniform})
which is stronger than the one known before.
\end{remark}

\begin{remark}
\label{rem:thm:Lojasiewicz:4}
It is not possible to obtain a uniform bound (i.e. a bound only in terms of $d$, $n$ and possibly the combinatorial parameters) on the constant $c$ in Theorem~\ref{thm:Lojasiewicz}. 
\end{remark}

As mentioned earlier Theorem~\ref{thm:Lojasiewicz} leads to improvements in several applications where {\L}ojasiewicz inequality plays an important role. 
We discuss some of these applications in depth in Section~\ref{sec:applications}, but mention an
important one right away.

\subsection{Application to error bounds}
\label{subsec:error-bounds}
Study of \emph{error bounds} (defined next) is a very important topic in optimization theory and computational optimization (see for example ~\cite{Pang97} and the references cited therein). 

\begin{definition}[Error bounds and residual function]
\label{def:eb}
Let $M,E \subset \R^n$.
An \emph{error bound on $E$ with respect to $M$} is an inequality
\begin{align}\label{eq:error_bound_generic_form}
\distance(x,M)^{\rho} \le \kappa \cdot \psi(x), \qquad \forall x \in E,
\end{align}
where $ \rho, \kappa > 0$, and $\psi:M \cup E \to \R_{\geq 0}$ is some function (called 
a \emph{residual function}) such that $\psi(x)=0$ iff $x \in M$.
\end{definition}
The study of error bounds was motivated by the implementation of iterative numerical optimization algorithms and the proximity of solutions to the feasible or optimal set. Thus, from the optimization point of view, the set $M$ in~\eqref{eq:error_bound_generic_form} can be the feasible set (polyhedron, a slice of the positive semi-definite cone, a basic semi-algebraic set, etc.) or the optimal set of an optimization problem, see~\eqref{poly_optim_def}, $E$ is a set of interest (e.g., iterates of an iterative algorithm or central solutions~\cite{BM22}), and a residual function $\psi$ measures the amount of violation of the inequalities defining $M$ at a given solution of $E$. See~\cite{Pang97} for other applications of error bounds in optimization.

If $M$ and $\psi$ are semi-algebraic and $E$ is a closed and bounded semi-algebraic set, 
then~\eqref{eq:error_bound_generic_form} is a special case of~\eqref{Orig_Lojasiewicz}, 
and by Theorem~\ref{thm:Lojasiewicz}, the error bound~\eqref{eq:error_bound_generic_form} exists with an integer $\rho \ge 1$. 

We prove the following quantitative result.

\begin{theorem}
\label{thm:error_bound_semialg_set}
Let $M$ be a basic closed semi-algebraic set defined by 
\begin{align}\label{eq:basic_semi_error_bound}
M:=\{x \in \R^n \mid g_i(x) \le 0, \ h_j(x) = 0, \ i=1,\ldots,r, \ j=1,\ldots,s\},
\end{align}
where $g_i,h_j \in \R[X_1,\ldots, X_n]_{\le d}$,
and let $E$ be a closed and bounded $\mathcal{P}$-semi-algebraic subset of $\R^n$ 
with $\mathcal{P} \subset \R[X_1,\ldots,X_n]_{\leq d}$.
Let $\psi: M \cup E \rightarrow \R_{\ge 0}$ be the semi-algebraic function defined by 
\begin{align}\label{func:residual_function}
\psi(x)=\sum_{j=1}^s |h_j(x)| + \sum_{i=1}^r \max\{g_i(x),0\}.
\end{align}
Then there exist a positive constant $\kappa$ and an integer $\rho \ge 1$ such that~\eqref{eq:error_bound_generic_form} holds, with $\rho=d^{O(n^2)}$. 

Moreover, if $ \dim M = 0$ (i.e. $M$ is a finite subset of $\R^n$), then 
\eqref{eq:error_bound_generic_form} holds with  
$\rho \le (8d)^{2(n+7)} = 
d^{O(n)}
$.
\end{theorem}

\begin{remark}
\label{rem:thm:error_bound_semialg_set}
Example~\ref{ex:exponential_dependence} indicates that 
the upper bound on $\rho$ cannot be better than $d^n$. 
\end{remark}

\begin{remark}\label{rem:arbitrary_residual_functions}
The error bounds of Theorem~\ref{thm:error_bound_semialg_set} have been stated, for the purpose of applications to optimization, only in reference to the basic semi-algebraic set~\eqref{eq:basic_semi_error_bound} and the residual function~\eqref{func:residual_function}. However, the results of Theorem~\ref{thm:error_bound_semialg_set} are still valid if we replace ~\eqref{eq:basic_semi_error_bound} by any $\mathcal{Q}$-semi-algebraic set $M$ and~\eqref{func:residual_function} by any $\mathcal{Q}'$-semi-algebraic residual function, where $\mathcal{Q} \subset \R[X_1,\ldots,X_n]_{\leq d}$ and $\mathcal{Q}' \subset \R[X_1,\ldots,X_n,Y]_{ \leq d^{O(n)}}$, see Corollary~\ref{cor:semi-definite_sys}. \hide{
For instance, one may define 
\begin{align*}
    \psi(x)= \sup_{y \in M'} f(x,y), 
\end{align*}
where $f \in \R[X_1,\ldots,X_n,Y_1,\ldots,Y_{\ell}]_{\le d}$, $\ell =O(n)$, and $M'$ is a closed and bounded  $\mathcal{Q}''$-semi-algebraic set with $\mathcal{Q}'' \subset \R[Y_1,\ldots,Y_{\ell}]_{\leq d}$. Notice that the fact that $M'$ is closed and bounded implies that $\psi$ is continuous.
These results 
should
be compared with~\cite[Corollary~4.4]{LMNP18} and are applicable to error bounds for polynomial semi-infinite and bi-level optimization problems.}
\end{remark}

\hide{
\begin{remark}
If $g_i,h_j \in \mathbb{Z}[X_1,\ldots, X_n]_{\le d}$, then the proof of Theorem~\ref{thm:Lojasiewicz} also allows for deriving an upper bound on $\kappa$, analogous to the one in~\cite[Theorem~7]{S91}.
\end{remark}
}
\noindent
Theorem~\ref{thm:Lojasiewicz} significantly improves the bound $D^{n^{c_1}}$ in~\cite[Theorem~7]{S91} in which $D$ is the sum of degrees of polynomials and $c_1$ is universal positive integer. Furthermore, the upper bound on $\rho$ in Theorem~\ref{thm:error_bound_semialg_set} is independent of $r$ and $s$ (the number of polynomial equations and inequalities in~\eqref{eq:basic_semi_error_bound}), which is particularly important for optimization purposes. Thus, if $r,s = \omega(n^2)$,
the first part of Theorem~\ref{thm:error_bound_semialg_set} improves the best current upper bound~\cite[Corollary~3.8]{LMP15}, where
\begin{align}\label{eqn:best_error_bound}
\rho \le \min\big\{(d+1)(3d)^{n+r+s-1},d(6d-3)^{n+r-1}\big\}.
\end{align}
\noindent
Furthermore, when $M$ is a finite subset of $\R^n$ and $r =\omega(n)$, the second part of Theorem~\ref{thm:error_bound_semialg_set} improves the upper bound~\cite[Theorem~4.1]{LMP15}, where  
\begin{align}\label{eqn:best_error_bound_compact_case}
    \rho \le \frac{(2d-1)^{n+r}+1}{2}.
\end{align}

The improvements mentioned above are particularly relevant to non-linear semi-definite systems, non-linear semi-definite optimization, and semi-definite complementarity problems, see e.g.,~\cite{LP98,GP02}, where the application of~\eqref{eqn:best_error_bound} and~\eqref{eqn:best_error_bound_compact_case} would result in a doubly exponential bound 
since the number of inequalities needed to define the cone of positive semi-definite matrices in 
the vector space $\mathbb{S}^n$ of symmetric $n \times n$ matrices with entries in $\R$
is exponential in $n$. 
The problem of estimation of the exponent $\rho$  in the error bounds of positive semi-definite systems failing the Slater condition~\cite[Page~23]{Kl02} is posed in \cite[Page~106]{LP98} where it is stated 
``Presently, we have no idea of what this exponent ought to be except in trivial cases''. 
Corollary~\ref{cor:semi-definite_sys} quantifies the error bound exponent in~\cite[Proposition~6]{LP98} and gives an answer to this question in the special case of polynomial mappings.

\begin{corollary}\label{cor:semi-definite_sys}
Let $\mathbb{S}_+^p$ be the cone of symmetric $p \times p$ positive semidefinite matrices (with entries in $\R$), let $M$ be defined as
\begin{align*}
    M:=\{X \in \mathbb{S}^p \mid \ g_i(X) \le 0, \ i=1,\ldots,r\},
\end{align*}
where $g_i:\mathbb{S}^p \to \R$ is a polynomial function of degree $d$, and let $E$ be a closed and bounded $\mathcal{P}$-semi-algebraic subset of $\R^{p^2}$ 
with $\mathcal{P} \subset \R[X_1,\ldots,X_{p^2}]_{\leq d}$.
Then there exist $\kappa \in \R $ and $\rho=\max\{d,p\}^{O(p^4)}$ such that
\begin{align*}
\distance(x,M \cap \mathbb{S}_+^p)^{\rho} \le \kappa \cdot \max\Big\{\distance(x,M), \max\{-\lambda_{\min}(x),0\}\Big\},
\quad \mbox{ for all } x \in E.
\end{align*}
\end{corollary}

\subsection{O-minimal case}
Many finiteness results of semi-algebraic geometry generalize to arbitrary o-minimal expansions
of $\mathbb{R}$ (we refer the reader to \cite{Dries} and \cite{Michel2} for the definition of
o-minimal structures and the corresponding finiteness results). {\L}ojasiewicz inequality does not extend to
arbitrary o-minimal expansions of $\mathbb{R}$ (for example, o-minimal expansions in which the exponential function
is definable). However,  Miller proved in \cite[Theorem 5.4, page 94]{Miller} that in a
\emph{polynomially bounded o-minimal expansion of $\mathbb{R}$}, 
{\L}ojasiewicz inequality holds for every compact definable
set $A$ and definable functions $f,g:A \rightarrow \mathbb{R}$ with $f^{-1}(0) \subset g^{-1}(0)$ (using the same notation as in Theorem~\ref{thm:Lojasiewicz} above).
(An o-minimal expansion of $\mathbb{R}$ is polynomially bounded if for every definable function $f:\mathbb{R} \rightarrow \mathbb{R}$, there exist
$N \in \mathbb{N}, c \in \mathbb{R}$, such that $|f(x)| <  x^N$ for all $x > c$.)
However, it is not possible to give a meaningful quantitative version of Miller's result in such a general context.

\subsubsection{Extension of the notion of combinatorial complexity to arbitrary o-minimal structures}
Although the notion of algebraic complexity in the context of general o-minimal structure does not make sense
in general -- one
can still talk of combinatorial complexity \cite{Basu9}. 
The following result is illustrative (see also Proposition~\ref{prop:qcad} for another example) 
and can be obtained by combining  \cite[Theorem 2.3]{Basu9}
and the approximation theorem proved in \cite{Gabrielov-Vorobjov}.

Fix an o-minimal expansion of the 
$\mathbb{R}$ and suppose that $\mathcal{A}$ is a definable family of closed subsets of $\mathbb{R}^n$. Then there exists 
a constant $C = C(\mathcal{A}) > 0$ having the following property. Suppose that
$\mathcal{S} \subset \mathcal{A}$ is a finite subset and $S$ a subset of $\mathbb{R}^n$ belonging to the Boolean algebra 
of subsets of $\mathbb{R}^n$ generated by $\mathcal{S}$ such that
\begin{equation}
\label{eqn:Betti-o}
\sum_i b_i(S) \leq C \cdot s^n,
\end{equation}
where $s = \card(\mathcal{S})$ and $b_i(\cdot)$ denotes the $i$-th Betti number~\cite{BPRbook2posted}. Notice that this bound does depend on the combinatorial parameter $s$.
Note also that it follows from Hardt's triviality theorem for o-minimal structures \cite{Michel2} that
the Betti numbers of the sets appearing in any definable family are bounded by a constant (depending on the family). 
However, the family of sets $S$ to which the inequality \eqref{eqn:Betti-o} applies is
not necessarily a definable family.

In view of the inequality \eqref{eqn:Betti-o} 
it is an interesting question whether one can prove a quantitative version of Miller's result 
with a uniform bound on the {\L}ojasiewicz exponent in the same setting as above -- so that the bound applies
to a family (not necessarily definable) of definable sets $S$ and functions $f,g$ simultaneously -- as
in inequality \eqref{eqn:Betti-o}.

\subsubsection{{\L}ojasiewicz inequality in polynomially bounded o-minimal structures} 
For the rest of this section we fix a \emph{polynomially bounded}
o-minimal expansion of $\mathbb{R}$. Before stating our theorem we need to use a notation and a definition. 

\begin{notation}
A definable family of subsets of $\mathbb{R}^n$ parametrized by a definable set 
$A$, is a definable subset $\mathcal{A} \subset A \times \mathbb{R}^n$. For $a \in A$, we will denote by
$\mathcal{A}_a = \pi_2(\pi_1^{-1}(a) \cap \mathcal{A})$, where $\pi_1: A \times \mathbb{R}^n \rightarrow A$, $\pi_2:A \times \mathbb{R}^n \rightarrow \mathbb{R}^n$ are the two projection maps. We will often abuse notation and refer to the definable family by $\mathcal{A}$.
\end{notation}

\begin{definition}
Given a definable family $\mathcal{A}$ of subsets of $\mathbb{R}^n$ parametrized by a definable set $A$,
and a finite subset $A' \subset A$,
we call a subset $S \subset \mathbb{R}^n$ to be a $(\mathcal{A},A')$-set if it belongs to the Boolean algebra
of subsets of $\mathbb{R}^n$ generated by the tuple $(\mathcal{A}_a)_{a \in A'}$. 
We will call a subset $S \subset \mathbb{R}^n$, a $\mathcal{A}$-set if $S$
is a $(\mathcal{A},A')$-set for some finite set $A' \subset A$. 
If the graph of a definable function $f$ is  a $\mathcal{A}$-set we will call $f$ a $\mathcal{A}$-function.
(Note that the
family of $\mathcal{A}$-sets is in general \emph{not} a definable family of subsets
of $\mathbb{R}^n$.)
\end{definition}

\begin{example}
\label{eg:sa}
If we take the o-minimal structure $\mathbb{R}_{\mathrm{sa}}$ of semi-algebraic sets,
then for each fixed $d$,
the family of semi-algebraic sets which are  
$\mathcal{P}$-semi-algebraic sets where $\mathcal{P}$ varies over all finite subsets of $\mathbb{R}[X_1,\ldots,X_n]_{\leq d}$
is an example of a family of $\mathcal{A}$-sets for an appropriately chosen $\mathcal{A}$. Note that this
family is not a semi-algebraic family.
\end{example}

Example~\ref{eg:sa} suggests a way to obtain a quantitative {\L}ojasiewicz inequality 
valid over any polynomially bounded o-minimal structure.

\begin{theorem}[{\L}ojasiewicz inequality for $\mathcal{A}$-sets and $\mathcal{B}$-functions for any pair of 
definable families $\mathcal{A}$ and $\mathcal{B}$]
\label{thm:uniform}
Let $\mathcal{A}$ be a definable family of subsets of $\mathbb{R}^n$ parametrized by the definable set $A$,
and let $\mathcal{B}$ be a definable family of subsets of $\mathbb{R}^{n+1}$ parametrized by the definable set $B$.

Then there exists $N = N(\mathcal{A},\mathcal{B}) > 0$ having the following property.
For any triple of finite sets
$(A',B',B'')$ with $A' \subset A, B',B'' \subset B$, 
there exists $c = c(A,B',B'') \in \mathbb{R}$ such that 
for each closed and bounded $(\mathcal{A},A')$-set $S$, 
a $(\mathcal{B},B')$-set $F$, and a $(\mathcal{B},B'')$-set $G$, such that $F,G$ are graphs of
definable functions $f,g:\mathbb{R}^n \rightarrow \mathbb{R}$ continuous on $S$ with 
$f|_S^{-1}(0) \subset g|_S^{-1}(0)$, 
and for all $x \in S$, 
\[
|g(x)|^N \leq c \cdot |f(x)|.
\]
\end{theorem}

\begin{remark}[Theorem~\ref{thm:uniform} generalizes Theorem~5.4 (page 94) in~\cite{Miller}]
Notice that as in Theorem~\ref{thm:Lojasiewicz}, the 
combinatorial parameter (namely, $\card(A' \cup B' \cup B'')$) plays no role. It is also more general than the 
{\L}ojasiewicz inequality in \cite{Miller} as the inequality holds with the same value of $N$ for a large,
potentially infinite family (not necessarily definable)  of triples $(S,f,g)$ and not just for one triple
as in the result in \cite{Miller}.
\end{remark}

\subsection{Outline of the proofs of the main theorems}
\label{subsec:proofs:outline}
We now outline the key ideas behind the proofs of the theorems stated in the previous section.

Our proof of Theorem~\ref{thm:Lojasiewicz} follows closely the proof of the similar qualitative statement in
\cite{BCR98} with certain important modifications. One main tool that we use to obtain our quantitative bound
is a careful analysis of the degrees of certain polynomials appearing in the output of an algorithm 
(Algorithm 14.6 (Block Elimination)
described in the book \cite{BPRbook2posted}. 
\footnote{We refer to the posted online version of the book because it
contains certain degree estimates which are more precise than in the printed version.}
This algorithm is an intermediate algorithm for the effective quantifier elimination algorithm described in
\cite{BPRbook2posted}, and takes as input a finite set of polynomials $\mathcal{P} \subset \R[Y,X]$, and 
produces as output a finite set $\BElim_X(\mathcal{P}) \subset \R[Y]$ having the property that
for each connected component $C$ of the realization of each realizable sign condition (see Notation~\ref{not:sign-condition}) the set of sign conditions realized by $\mathcal{P}(y,X))$ is constant
as $y$ varies over $C$.
The precise mathematical statement describing the above property of the output of the Block Elimination algorithm  (including a
bound on the degrees of the polynomials output) is summarized  in Proposition~\ref{prop:BElim}. 
The proof of the bound on degrees borrows heavily from the complexity analysis of the algorithm that already appears in \cite{BPRbook2posted} but with an added part corresponding to the last step of the algorithm. This last step uses another algorithm (namely, Algorithm 11.54 (Restricted Elimination) in \cite{BPRbook2posted}) and we use the complexity analysis of this algorithm as well.

We also need a quantitative statement on the growth of a semi-algebraic function of one variable at infinity whose graph is defined by polynomials of a given degree. It is important for us that the growth is bounded only by the
upper bound on the degree and not on the size of the formula describing the graph. This is proved in Lemma~\ref{lem:2dim}.

Note that the technique of utilizing complexity estimates of algorithms to prove 
quantitative bounds in real algebraic geometry is not altogether new. For example, similar ideas have been used to prove a quantitative curve selection lemma \cite{BR2021}, 
and bounds on the radius of a ball guaranteed to intersect
every connected component of a given semi-algebraic set \cite{BR10}, amongst other such results.

Theorem~\ref{thm:error_bound_semialg_set} is an application of Theorem~\ref{thm:Lojasiewicz} with 
one key extra ingredient. We prove (see Lemma~\ref{upper_bound_dist_func})
that if $M$ is a $\mathcal{P}$-semi-algebraic set with
$\mathcal{P} \subset \R[X_1,\ldots,X_n]_{\leq d}$, then the graph of the distance function, $\dist(\cdot,M)$, can be described by 
a quantifier-free formula involving polynomials having degrees at most $d^{O(n)}$. A more naive approach 
(for example that in \cite{S91})
involving quantifier elimination would involve elimination of two blocks of quantified variables with a quantifier alternation
and would lead to a degree bound of $d^{O(n^2)}$. The formula describing the graph of $\dist(\cdot,M)$ that we obtain is very large 
from the point of view of combinatorial complexity (compared to the more naive approach), but with a better degree
bound. We leverage now the fact that our bound on the {\L}ojasiewicz exponent is independent of the combinatorial parameter, and apply Theorem~\ref{thm:Lojasiewicz} to obtain the stated result.

The proof of Theorem~\ref{thm:uniform} is similar to that of proof Theorem~\ref{thm:Lojasiewicz} with the
following important difference. Proposition~\ref{prop:BElim} which plays a key role in the proof of  Theorem~\ref{thm:Lojasiewicz} is replaced by a quantitative version of the existence theorem for cylindrical definable decomposition
adapted to finite sub-families of a family $\mathcal{F}$  of definable subsets of $\R^n$ in any o-minimal
structure.
The important quantitative property that
we need is not the size of the decomposition but the fact that each cell of the decomposition is determined in
a certain fixed definable way from a certain finite number, $N(n)$, of the sets  of the given finite sub-family of $\mathcal{F}$ (the key point being that the number $N(n)$ is independent of the cardinality of the finite sub-family).
The existence of such decompositions in o-minimal structures was first observed in \cite[Theorem 2.5]{Basu9} 
(see Proposition~\ref{prop:2.6.4:o} below)
and it is closely related (in fact equivalent)
to the fact that o-minimal structures are \emph{distal} in the sense of model theory (see \cite{Starchenko-Bourbaki}). 

The rest of the paper is organized as follows. In Section~\ref{sec:prior} we survey prior work on 
proving bounds on the {\L}ojasiewicz exponent at various levels of generality,
and also survey prior work  on proving error bounds.
In Section~\ref{sec:comb-alg}, in order to put the current paper in context, 
we include a discussion of the role that the separation of combinatorial and algebraic
complexity has played in quantitative real algebraic geometry. 
In Section~\ref{sec:proofs} we prove the main theorems after introducing the necessary definitions and preliminary results.
In Section~\ref{sec:applications}, we discuss some further applications of our main theorem. Finally, in
Section~\ref{sec:conclusion} we end with some open problems.

\section{Prior and related work}
\label{sec:prior}

\subsection{Prior results on {\L}ojasiewicz inequality}
\label{subsec:prior-sa}
Solern\'o~\cite[Theorem~3]{S91} proved that~\eqref{Orig_Lojasiewicz} holds with $\rho = D^{c_1 n}$, in which $c_1$ is a universal constant and $D$ is an upper bound on the sum of the  degrees of polynomials in $\mathcal{P}$ and $\mathcal{Q}$. Since $D$ is an upper bound on the sum of the  degrees of polynomials in $\mathcal{P}$ and $\mathcal{Q}$, the bound in \cite[Theorem~3]{S91} depends implicitly on the cardinalities of
$\mathcal{P}$ and $\mathcal{Q}$ (unlike Theorem~\ref{thm:Lojasiewicz}).
In the case of polynomials with integer coefficients,  
Solern\'o~\cite[Theorem~3 (ii)]{S91} also proves an upper bound of
$2^{\tau{D^{c_2 \cdot n^2}}}$ on the constant $c$ (following the same notation as in \eqref{Orig_Lojasiewicz}) where $D$ is an upper bound on the sum of the  degrees of polynomials in $\mathcal{P}$ and $\mathcal{Q}$ and $c_2$ is a universal constant. This bound should be compared with inequality~\eqref{eqn:thm:Lojasiewicz:1} in Theorem~\ref{thm:Lojasiewicz}.

\vspace{5px}
\noindent
In a series of papers~\cite{KS14,KS16,KSS19} 
Kurdyka, Spodzieja and Szlachci{\'{n}}ska
proved several quantitative results on {\L}ojasiewicz inequality. We summarize them as follows.
Let $S \subset \mathbb{R}^n$ be a closed semi-algebraic set, and let 
\begin{align*}
    S = S_1 \cup \cdots \cup S_k
\end{align*}
be a decomposition~\cite{BCR98} of $S$ into $k$ closed basic $\mathcal{P}_i$-semi-algebraic subsets $S_i$ with $\mathcal{P}_i \subset \mathbb{R}[X_1,\ldots,X_n]_{\le d_i}$, each involving $r_i$ polynomial inequalities. Let $r(S)$ be the minimum of $\max\{r_1,\ldots,r_k\}$ over all possible decompositions of $S$ and let $\deg(S)$ denote the minimum of $\max\{d_1,\ldots,d_k\}$ over all decompositions for which $r_i \le r(S)$. Further, let $F: S \to \mathbb{R}^s$ be a continuous semi-algebraic mapping and suppose that $\mathbf{0} \in S$ and $F(\mathbf{0})= \mathbf{0}$. Then for every $\varepsilon > 0$ there exist~\cite[Corollary~2.2]{KS16} (see also~\cite{KSS19}) $c > 0$ and an integer
\begin{align}\label{Kurdyka_upper_bound}
1 \le \rho \le \bar{d}(6\bar{d}-3)^{n+s+\bar{r}-1}
\end{align}
such that  
\begin{align}\label{ineq:Kurdyka_bound}
    \|F(x)\| \ge c \cdot \distance\big(x,F^{-1}(\mathbf{0}) \cap S\big)^{\rho}, \quad \forall x \in S, \ \|x\| < \varepsilon,
\end{align}
where
\begin{align*}
\bar{d}&=\max\big\{\deg(S),\deg(\mathrm{graph}(F))\big\},\\
\bar{r} &= r(S) + r(\mathrm{graph}(F)).
\end{align*}
If $x=\mathbf{0}$ is an isolated zero of $F$, then $\rho \le ((2\bar{d}-1)^{n+s+\bar{r}}+1)/2$. 

Note that the above  bounds do depend on the number of polynomials.
Also, notice that $\bar{d}$ and $\bar{r}$ in the upper bound~\eqref{Kurdyka_upper_bound} are both different from $d$ and the number of inequalities in the semi-algebraic description of $f$, $g$, and $A$ in Theorem~\ref{thm:Lojasiewicz}. In fact, $\bar{r}$ and $\bar{d}$ is the number of inequalities and the maximum degree of polynomials in the minimal semi-algebraic description of $\mathrm{graph}(F)$ and $S$. It was proved in~\cite{Br91} that $\bar{r}$ is bounded by $n(n+1)/2$, but it is not clear how $d$ blows up for a minimal decomposition. 
Because of this we cannot directly compare the bound in~\eqref{Kurdyka_upper_bound} to 
that of Theorem~\ref{thm:Lojasiewicz} proved in the current paper.

\hide{
See also~\cite{RS11} for the local {\L}ojasiewicz exponent of a complex polynomial mapping.}

Let $f:S \to \mathbb{R}$ be a Nash function~\cite[Definition~2.9.3]{BCR98}, where $S$ is a compact semi-algebraic subset of $\mathbb{R}^n$. Osi\'nska-Ulrych et al.~\cite{OSS21} showed that 
\begin{align*}
    |f(x)| \ge c \cdot \distance\big(x,f^{-1}(0)\big)^{2(2\bar{d}-1)^{3n+1}}, \quad \forall x \in S,
\end{align*}
in which $\bar{d}=\deg_S(f):=\max\{\deg_a(f) \mid a \in S\}$, and $\deg_a(f)$ is the degree of the unique irreducible $P \in \mathbb{R}[X_1,\ldots,X_n,Y]$ such that $P(x,f(x))=0$ for all $x$ in a connected neighborhood of $a$.

\hide{
If $f$ is real analytic in a neighborhood of $0 \in \mathbb{R}^n$ with $f(0) = 0$ and $\nabla f(0) = 0$, Pham~\cite{P12} showed that there exist $c,\varepsilon > 0$ $\distance(x,f^{-1}(0))^{\varrho} \le c \cdot |f(x)|$ for all $\|x\| \le r$, where
\begin{align*}
    \varrho \le \max\{d(3d-4)^{n-1},2d(3d-3)^{n-2}\}.
\end{align*}
}
Koll\'ar \cite{K99} considered the problem of improving Solern\'o's results \cite{S91}. 
He obtained significant improvements but under certain restrictions.
More precisely, 
given a semi-algebraic set $M$ as in~\eqref{eq:basic_semi_error_bound}, 
with $\max_i\{f_i(x)\} > 0$ for all $x \in M$ for $0 < \|x\| \ll 1$, 
and $\max_i\{f_i(\mathbf{0})\} = 0$,
he proved that (\cite[Theorem~4]{K99})  there exist constants $c,\varepsilon > 0$ such that 
\begin{align}\label{Kollar_Lojasiewicz}
\max_i\{f_i(x)\} \ge c \cdot \|x\|^{B(n-1)d^n} \ \ \text{for all} \ x \in M \ \text{with} \ \|x\| < \varepsilon, 
\end{align}
where $B(n):=\binom{n}{\lfloor (n/2) \rfloor}$.
Notice that the bound $B(n-1)d^n \leq (2 d)^n$ on the exponent in \eqref{Kollar_Lojasiewicz} is a little better than the bound in Theorem~\ref{thm:Lojasiewicz}.
It is also the case that similar to our result,  Koll\'ar's bound is independent of the combinatorial parameters
(i.e. number of polynomials occurring in the definition of $M$ and the number of $f_i$'s). However, 
unlike Theorem~\ref{thm:Lojasiewicz}, the pair of functions
$\max_i\{f_i(x)\}, ||\cdot||$ in Koll\'ar's theorem is quite restrictive and so inequality
\eqref{Kollar_Lojasiewicz} is difficult to apply directly -- for instance, in applications to error bounds
considered in this paper (Theorem~\ref{thm:error_bound_semialg_set}). Moreover, the restriction that
$\mathbf{0}$ has to be an isolated zero of  $\max_i\{f_i(x)\}$ may not be satisfied in many applications, restricting
the utility of  \eqref{Kollar_Lojasiewicz}.

More recently, Osi\'nska-Ulrych et al.~\cite{OSS21b} proved that $\mathcal{L}(f,g \mid \mathbb{B}^n) \le \bar{d}^{4n+1}$ in which $\bar{d}$ is the degree of polynomials describing $f$ and $g$,
$\mathbb{B}^n$ is the unit ball in $\mathbb{R}^n$,
and $\bar{d}$ is a bound on the degrees of polynomials defining the graphs of $f,g$ as well as on certain
polynomials giving a suitable semi-algebraic decomposition of $\mathbb{B}^n$ adapted to $f,g$. 
Note that $\bar{d}$ could be larger than the degrees of the polynomials defining $f,g$.
This bound is also independent of the combinatorial parameters
but asymptotically weaker than the one in Theorem~\ref{thm:Lojasiewicz}, and the setting is more restrictive
(since the bound is not directly in terms of the degrees of the polynomials appearing in the definition of $f,g$).

\subsection{Other forms of {\L}ojasiewicz inequality}
Several other  forms of {\L}ojasiewicz inequality have appeared in the literature. Let $f:U \to \mathbb{R}$ be a real analytic function, where $U \subset \mathbb{R}^n$ is neighborhood of $\mathbf{0} \in \mathbb{R}^n$. If $f(\mathbf{0}) = 0$ and $\nabla f(\mathbf{0}) = \mathbf{0}$, then there exist a neighborhood $U'$ of $0$, a rational number $\varrho < 1$, and $c >0$ such that
\begin{equation}\label{gradLojasiewicz}
    |\nabla f(x)| \ge c \cdot|f(x)|^{\varrho}, \quad \forall x \in U',
\end{equation}
which is known as \textit{{\L}ojasiewicz gradient inequality}. The infimum of $\varrho$ satisfying~\eqref{gradLojasiewicz} is called the {\L}ojasiewicz exponent of $f$, and it is denoted by $\varrho(f)$. If $f\in \mathbb{R}[X_1,\ldots,X_n]_{\le d}$ and has a isolated singularity at zero, then there exists an upper bound~\cite{G99}
\begin{align*}
\varrho(f) \le 1-\frac{1}{(d-1)^n+1}.
\end{align*}
Under a weaker condition of having non-isolated singularity at the origin, D'Acunto and Kurdyka showed~\cite{DK05} that 
\begin{align}\label{eqn:Kurdyka_gradient_ineq}
\varrho(f) \le 1-\frac{1}{\max\{d(3d-4)^{n-1},2d(3d-3)^{n-2}\}}.
\end{align}

\vspace{5px}
\noindent
A more general result is given by~\cite{OSS21} for a Nash function $f:U \to \mathbb{R}$, where $U$ is a connected neighborhood of $\mathbf{0} \in \mathbb{R}^n$. If $f(\mathbf{0}) = 0$ and $\nabla f(\mathbf{0}) = \mathbf{0}$, then~\eqref{gradLojasiewicz} holds with 
\begin{align*}
\varrho(f) \le 1-\frac{1}{2(2d-1)^{3n+1}},
\end{align*}
where $d$ is the degree of the unique irreducible $P \in \mathbb{R}[X_1,\ldots,X_n,Y]$ such that $P(x,f(x))=0$ for all $x \in U$. If, in addition to the latter condition, $\frac{\partial P}{\partial y} (x, f (x)) \neq  0$ for all $x \in U$, then there exists a stronger upper bound   
\begin{align*}
\varrho(f) \le 1-\frac{1}{\max\{2d(2d-1),d(3d-2)^n\} +1} .
\end{align*}

\hide{
\vspace{5px}
\noindent
The complex case can be easier since polynomials have more complex zeros than reals. Let $V:=\zero(f_1,\ldots,f_m)$ be the zero set of polynomials $f_1,\ldots,f_m \in \mathbb{C}[X_1,\ldots,X_n]_{\le d}$. Then there exits~\cite{B88} a constant $c>0$ such that
\begin{align*}
    \min\{\distance(x,V),1\}^{\rho} \le c \cdot \max_i\{|f_i(x)|\},
\end{align*}
where $\rho = (n+1)^2d^{\min\{n,m\}}$. Note that the inequality~\eqref{Orig_Lojasiewicz} can be applied to this case by considering $V$ as a semi-algebraic subset of $\mathbb{R}^{2n}$. Let $f_1,\ldots,f_m \in \mathbb{K}[X_1,\ldots,X_n]_{\le d}$, where $\mathbb{K}$ is an algebraically closed field of any characteristic, and let the distance on $V$ be defined by 
\begin{align*}
    \distance(x,V)=\inf_{y \in V} \|x-y\|,
\end{align*}
where
\begin{align*}
    \|(x_1,\ldots,x_n)\|:=\sqrt{|x_1|^2+\ldots + |x_n|^2}
\end{align*}
in which $|\cdot|:\mathbb{K} \to [0,\infty)$ is a valuation satisfying the triangle inequality. Then there exists~\cite[Theorem~5]{JKS92} a constant $c > 0$ such that 
\begin{align*}
    \distance(x,V)^{\rho} \le c \cdot \max_i\{|f_i(x)|\},
\end{align*}
where $\rho = O(2^m d_1d_2\ldots d_m)$ for all $x$ in a compact subset of $\mathbb{K}^n$. See also~\cite[Theorem~3.2]{G95} for an estimate of {\L}ojasiewicz exponent for Pffafian equations.
}

\subsection{Prior work on error bounds}
\label{subsec:prior-error-bounds}
\hide{
The error bound~\eqref{eq:error_bound_generic_form} is called Lipschitzian if $\rho =1$, and H\"{o}lderian otherwise. In general, a Lipschitzian error bound may fail to exist, as can be seen in Example~\ref{ex:exponential_dependence} (even if $S$ is convex, a Lipschitzian error bound may fail to exist, see~\cite[Example~3]{LP98}). See also~\cite[Examples~1 and~2]{LP98} for cases where an error bound, even a H\"{o}lderian one, fails to hold for $E = \mathbb{R}^n$.

The convexity in error bounds has been fully exploited in the literature. If we let $\R=\mathbb{R}$ and $M$ be convex defined by quadratic inequalities%
\footnote{There are also Lipschitzian error bounds for convex systems where $g_i$ and $h_j$ are convex differentiable but not polynomials, see e.g.,~\cite{M85,R75}.}, i.e., $s = 0,$ then a H\"olderian error bound exists with $E=\mathbb{R}^n$ and $\rho \le 2^{r+1}$~\cite[Theorem 3.1]{WP94}. A Lipschitzian error bound is obtained~\cite{Hoff52,LL94} if there exists $x^{\circ} \in \mathbb{R}^n$ satisfying 
\begin{align}\label{Slater}
g_i(x^{\circ}) < 0, \quad i=1,\ldots,r, \qquad h_j(x^{\circ}) = 0, \quad j=1,\ldots,s.
\end{align}
\noindent
In case that the nonempty set $M$ is defined by linear inequalities only, a stronger error bound exists with $\rho=1$ and $E=\mathbb{R}^n$~\cite{Hoff52}.
} 

\vspace{5px}
\noindent
Error bounds were generalized to analytic systems and basic semi-algebraic sets in~\cite[Theorem 2.2]{LP94} and~\cite[Theorem 2.2]{LL94} based on the analytic form of {\L}ojasiewicz inequality and H\"ormander's results~\cite{H58} but without explicit information about the exponent. Recently, an explicit upper bound with respect to $M$, defined in~\eqref{eq:basic_semi_error_bound}, was given in~\cite{LMP15} which depends exponentially on the dimension and the number of polynomial equations and inequalities, see~\eqref{eqn:best_error_bound}. The upper bound~\eqref{eqn:best_error_bound} follows from the bound~\eqref{eqn:Kurdyka_gradient_ineq} and the generalized differentiation in variational analysis. 

\vspace{5px}
\noindent
\hide{There are special cases of $M$ which yield a stronger lower bound than~\eqref{eqn:best_error_bound} and Theorem~\ref{thm:error_bound_semialg_set}. As a nontrivial special case of a semi-algebraic set, a linear matrix inequality system in conic form is defined as $M=\{y \in \mathbb{R}^m \mid C-\sum_{i=1}^m A^i y_i \succeq 0\}$, where $C,A^i$ for $i=1,\ldots,m$ are symmetric matrices. If $\interior(\{y \mid -\sum_{i=1}^m A^i y_i \succeq 0\}) \neq \emptyset$ (which also implies $\interior(M) \neq \emptyset$), then there exists~\cite[Corollary~2.1]{DH99} $\kappa > 0$ such that
 \begin{align*}
 \distance(y, M) \le \kappa \cdot \big(\max\{-\lambda_{\min}(C- \sum_{i=1}^n A^i y_i ),0\}\big), \qquad \forall y \in \mathbb{R}^m,
 \end{align*}  
see also~\cite[Proposition~5]{LP98}. More generally, even if $\interior(M) = \emptyset$, for every compact set $E \subset \mathbb{R}^m$ there exist~\cite[Theorem 3.3]{ST2000} $\kappa, \rho > 0$ such that
 \begin{align}\label{error_bound_SDO}
 \distance(y, M)^{\rho} \le \kappa \cdot \big(\max\{-\lambda_{\min}(C- \sum_{i=1}^n A^i y_i ),0\}\big), \qquad \forall y \in E,
 \end{align}
where $\rho \le 2^{n-1}$.} 

\vspace{5px}
\noindent
In case that $M$ is defined by a single convex polynomial inequality, i.e., $r=1$, $s=0$, then~\eqref{eq:error_bound_generic_form} holds with $\rho \le (d-1)^n +1$~\cite[Theorem~4.2]{Li10}. Additionally, if there exists an $x \in M$ such that $g(x) < 0$, then $\rho = 1$ with $E=\mathbb{R}^n$~\cite[Theorem~4.1]{Li10}. More generally, there exists~\cite[Corollary~3.4]{BLY14} 
\begin{align}\label{upper_bound_convex_semi-algebraic}
\rho \le \min\Bigg\{\frac{(2d-1)^n+1}{2}, \binom{n-1}{[(n-1)/2]} d^n\Bigg\}    
\end{align}
such that the error bound~\eqref{eq:error_bound_generic_form} holds, where
\begin{align*}
\psi(x)= \max_{i \in \{1,\ldots,r\}} (\max\{g_i(x),0\}).
\end{align*}

\vspace{5px}
\noindent
A complete survey of error bounds in optimization and their applications to algorithms and sensitivity analysis can be found in~\cite{LP98,Pang97}.

\section{Combinatorial and algebraic complexity}
\label{sec:comb-alg}
A key feature of the bound in Theorem~\ref{thm:Lojasiewicz} is that it is \emph{independent of the
cardinality of $\mathcal{P}$ and $\mathcal{Q}$} and depends only on the bound on the maximum degree
of the polynomials in $\mathcal{P} \cup \mathcal{Q}$ and $n$.
In fact in many quantitative results (upper bounds on various quantities) in real algebraic geometry involving a $\mathcal{P}$-semi-algebraic set a fruitful distinction can be made between the dependence of the bound
on the cardinality of the set $\mathcal{P}$ and on the maximum degrees (or some other measure of the complexity) of the polynomials in $\mathcal{P}$. The former is referred to as the \emph{combinatorial part} and the latter as the
\emph{algebraic part} of the bound (see \cite{B17}). This distinction is important in many applications (such as in discrete and computational geometry) where the algebraic part of the bounds are treated as bounded by a fixed constant and only the combinatorial part is considered interesting.  

The following examples illustrate the different nature of the dependencies
on the combinatorial and the algebraic parameters in quantitative bounds appearing in real algebraic geometry
and put in context the key property of Theorem~\ref{thm:Lojasiewicz} stated in the beginning of this subsection. 

\begin{enumerate}[1.]
    \item (Bound on Betti numbers.)  Suppose that $S \subset \R^n$ is a $\mathcal{P}$- semi-algebraic set and $V = \ZZ(Q,\R^n)$
    a real algebraic set. Suppose that $\dim_\R V = p$, and the degrees of $Q$ and the polynomials in $\mathcal{P}$ is bounded by $d$. Then,
    \begin{equation}
    \label{eqn:Betti}
    \sum_i b_i(S \cap V) \leq s^p (O(d))^n, 
    \end{equation}
    where $s = \card(\mathcal{P})$ and $b_i(\cdot)$ denotes the $i$-th Betti number \cite{BPRbook2posted}. 
    Notice the different dependence of the bound on $s$ and $d$. 
    
    \item (Quantitative curve selection lemma.)
    The curve selection lemma \cite{Loj1,Loj3} (see also \cite{Milnor3}) is a fundamental result in semi-algebraic geometry. 
    \hide{
    \begin{theorem*}[Curve Selection Lemma]
\label{csl}
  Let $S \subset \R^k$ be a semi-algebraic set and $x \in \bar{S}$. Then there
  exist a positive element $t_0$ of $\R$, and a semi-algebraic path $\varphi$ from
  $[0, t_0)$ to $\R^k$ such that $\varphi (0) = x$ and $\varphi ((0, t_0)) \subset S$.
\end{theorem*}
}
The following quantitative version of this lemma was proved in \cite{BR2021}. 
    
    \begin{theorem*}[Quantitative Curve Selection Lemma]
\label{qcsl}
  Let  $\mathcal{P} \subset \R[X_1,\ldots,X_n]_{\leq d}$ be a finite set, 
  $S$ a $\mathcal{P}$-semi-algebraic set,  and $x \in \overline{S}$.
  Then, there exist $t_0\in \R, t_0 > 0$,
  a semi-algebraic path $\varphi: [0,t_0) \rightarrow \R^n$ with  
  \[
  \varphi(0) = x, \ \varphi ((0,t_0)) \subset S,
  \]
  such that the degree of the Zariski closure of the image of $\varphi$ is bounded by
  \[
  (O(d))^{4n+3}.
  \]
\end{theorem*}
Notice that the bound on the degree of the image of the curve $\varphi$  
in the above theorem has no combinatorial part, i.e.,
there is no dependence on the cardinality of $\mathcal{P}$ (unlike the bound in \eqref{eqn:Betti}). 

\item (Effective quantifier elimination)
Quantifier elimination is a key property of the theory of the reals, and has been studied widely
from the complexity view-point. The following quantitative version appears in \cite{BPR95}. 
\begin{theorem}
[Quantifier Elimination] \label{14:the:tqe}
Let $\mathcal{P} \subset \R[X_{[1]},\ldots, X_{[\omega]}, Y]_{\leq d}$ be a
finite set of $s$ polynomials, where $X_{[i]}$ is a block of $k_i$ variables, and
$Y$ a block of $\ell$ variables.
Let 
\[
\Phi (Y) = (Q_1 X_{[1]})\cdots (Q_\omega X_{[\omega]}) \Psi(X_{[1]},\ldots, X_{[\omega]}, Y)
\]
be a quantified-formula, with $Q_i \in \{\exists,\forall\}$ and $\Psi$ a $\mathcal{P}$-formula.
 Then there exists a quantifier-free formula 
  \[ \Psi (Y) = \bigvee_{i=1}^{I} \bigwedge_{j=1}^{J_{i}} (
     \bigvee_{n=1}^{N_{ij}} \sign (P_{ijn} (Y))= \sigma_{ijn} ) , \]
  where $P_{ijn} (Y)$ are polynomials in the variables $Y$, $\sigma_{ijn} \in
  \{0,1, - 1\}$,
  \begin{eqnarray*}
    I & \leq & s^{(k_{\omega} +1) \cdots (k_{1} +1) ( \ell +1)} d^{O
    (k_{\omega} ) \cdots O (k_{1} ) O ( \ell )} ,\\
    J_{i} & \leq & 
      s^{(k_{\omega} +1) \cdots (k_{1} +1)} d^{O (k_{\omega} ) \cdots O (k_{1}
      )}
    ,\\
    N_{ij} & \leq &
      d^{O (k_{\omega} ) \cdots O (k_{1} )}
  ,
  \end{eqnarray*}
  equivalent to $\Phi$,
  and the degrees of the polynomials $P_{ijk}$ are bounded by $d^{O
  (k_{\omega} ) \cdots O (k_{1} )}$. 
  
  Moreover, if the polynomials in $\mathcal{P}$ have coefficients in $\mathbb{Z}$ with bit sizes bounded by $\tau$, the polynomials $P_{ijk}$ also have integer coefficients with 
  bit-sizes bounded by $\tau d^{O (k_{\omega} ) \cdots O (k_{1} )}$. 
\end{theorem}
Notice that the bound on the degrees of the polynomials appearing in the quantifier-free formula is independent
of the combinatorial parameter $s =\card(\mathcal{P})$. This fact will play a key role in the 
proof of the main theorem (Theorem~\ref{thm:Lojasiewicz}) below.
\end{enumerate}

\section{Proofs of the main results}
\label{sec:proofs}

\subsection{Proof of Theorem~\ref{thm:Lojasiewicz}}
\label{subsec:proof:thm:Lojasiewicz}
Before proving Theorem~\ref{thm:Lojasiewicz} we need some preliminary results.

\subsubsection{Cylindrical definable decomposition}

The notion of cylindrical definable decomposition (introduced by {\L}ojasiewicz \cite{Loj1,Loj2}) plays a very important in semi-algebraic
and more generally o-minimal geometry \cite{Michel2}. We include its definition below for the sake of completeness  
and also for fixing notation that will be needed later.

\begin{definition}[Cylindrical definable decomposition]
\label{def:cad}
Fixing the standard basis of $\R^n$, we identify for each $i,1 \leq i \leq n$, $\R^i$ with the span of the first
$i$ basis vectors.
Fixing an o-minimal expansion of $\R$,
a cylindrical definable decomposition  of $\R$ is an $1$-tuple $(\mathcal{D}_1)$, where
$\mathcal{D}_1$ is a finite set of subsets of $\R$, each element being a point or an open interval,
which together gives a partition of $\R$. 
A cylindrical definable decomposition of $\R^n$ is an $n$-tuple $(\mathcal{D}_1,\ldots,\mathcal{D}_n)$,
where each $\mathcal{D}_i$ is a decomposition of $\R^i$,
$(\mathcal{D}_1,\ldots, \mathcal{D}_{n-1})$ is a cylindrical decomposition of $\R^{n-1}$,
and $\mathcal{D}_n$ is a finite set of definable subsets of $\R^n$ (called the cells of $\mathcal{D}_n$) giving a partition of $\R^n$ consisting 
of the following:
for each $C \in \mathcal{D}_{n-1}$, there is a finite set of definable continuous functions
$f_{C,1}, \ldots, f_{C,N_C}: C \rightarrow \R$ such that
$f_{C_1} < \cdots < f_{C,N_C}$,
and each element of $\mathcal{D}_n$ is either the graph of a function $f_{C,i}$ 
or of the form 
\begin{enumerate}[(a)]
    \item $\{ (x,t) \;\mid\; x\in C, t < f_{C,1}(x) \}$, 
    
    \item  $\{ (x,t) \;\mid\; x\in C, f_{C,i}(x) < t < f_{C,i+1}(x) \}$, 
    
    \item  $\{ (x,t) \;\mid\; x\in C, f_{C,N_C}(x) < t  \}$, 
    
    \item  $\{ (x,t) \;\mid\; x\in C \}$
\end{enumerate}
(the last case arising is if the set of functions $\{f_{C,i} | 1 \leq i \leq N_C \}$ is empty).

We will say that the cylindrical definable decomposition $(\mathcal{D}_1,\ldots,\mathcal{D}_n)$ is adapted to 
a definable subset $S$ of $\R^n$, if for each $C \in \mathcal{D}_n$, $C \cap S$ is either equal to 
$C$ or empty.

In the semi-algebraic case we will refer to a cylindrical definable decomposition by 
cylindrical algebraic decomposition.

\end{definition}

In the semi-algebraic case we will use the following extra notion.

\subsubsection{Sign conditions}
\begin{notation}[Sign conditions and their realizations]
\label{not:sign-condition}
Let $\mathcal{P}$ be a finite subset of $\R[X_1,\ldots,X_n]$.
For $\sigma \in \{0,1,-1\}^{\mathcal{P}}$, we call the formula $\bigwedge_{P \in \mathcal{P}} (\sign(P) = \sigma(P))$
to be a 
\emph{sign condition on $\mathcal{P}$} and call its realization the \emph{realization of the sign condition $\sigma$}.
We say that a \emph{sign condition is realizable} if its realization is not empty.

We denote by 
$\Cc(\mathcal{P})$ the set of semi-algebraically connected components of the realizations of each realizable sign condition on
$\mathcal{P}$.

We say that a cylindrical algebraic decomposition $\mathcal{D} = (\mathcal{D}_1,\ldots, \mathcal{D}_n)$ 
of 
$\R^n$ is \emph{adapted to $\mathcal{P}$} if for each
cell $C$ of $\mathcal{D}_n$, and each $P \in \mathcal{P}$, $\sign(P(x))$ is constant
for $x \in C$. (This implies in particular each element of $\Cc(\mathcal{P})$ is a union of cells of 
$\mathcal{D}_n$.)
\end{notation}

\begin{lemma}
\label{lem:2dim}
Let $\mathcal{F} \subset \R[X_1,X_2]_{\leq p}$ be a finite set of non-zero polynomials.
Let $f:\R \rightarrow \R$ be a semi-algebraic map,
such that 
\[
\mathrm{graph}(f) = \{(x,f(x)) \mid x \in \R\} = \bigcup_{C \in \mathcal{C}} C
\]

for some subset $\mathcal{C} \subset \Cc(\mathcal{F})$.

Then, there exist $a, c \in \R$ such that for all $x \geq a$,
\[
|f(x)| \leq c \cdot x^p.
\]
\end{lemma}

\begin{proof}
Consider a cylindrical algebraic decomposition $\mathcal{D} = (\mathcal{D}_1,\mathcal{D}_2)$ of $\R^2$ (with coordinates $X_1,X_2$) adapted to 
the set $\mathcal{F}$.

This implies that each $C \in \Cc(\mathcal{F})$ is a union of cells of $\mathcal{D}_2$.
Let $a_0 < a_1 <\cdots < a_n=a$ be the end-points of the intervals giving the partition of
$\R$ (corresponding to the $X_1$ coordinate)  in the decomposition $\mathcal{D}_1$.

Since the cylindrical decomposition $\mathcal{D}$ is adapted to $\mathcal{F}$,
and 
$\mathrm{graph}(f)
=\bigcup_{C \in \mathcal{C}} C
$
for some subset $\mathcal{C} \subset \Cc(\mathcal{F})$,
$\dim(C) \leq 1$ for each $C \in \mathcal{C}$, since 
\[
\dim(C) \leq \dim (\mathrm{graph}(f)) = 1.
\]
Hence, there exists for each $C \in \mathcal{C}$
a polynomial $F \in \mathcal{F}$, such that $F(x) = 0$ for all $x \in C$.

Also, since $\mathrm{graph}(f)
=\bigcup_{C \in \mathcal{C}} C
$
and each $C \in \mathcal{C}$ is a union of cells of $\mathcal{D}$, 
there exists a continuous semi-algebraic function $\gamma: (a,\infty) \rightarrow \R$, such that
$\mathrm{graph}(\gamma) \subset \mathrm{graph}(f)$, 
and $\mathrm{graph}(\gamma)$ is a cell of  $\mathcal{D}_2$.

Let $C \in \Cc(\mathcal{C})$ be the unique  element of $\mathcal{C}$ which contains 
$\mathrm{graph}(\gamma)$, and $F \in \mathcal{F}$ such that $F(x) = 0$ for all $x \in C$.

The lemma now follows from \cite[Proposition 2.6.1]{BCR98} noting that 
\[
\deg(F) \leq p.
\]
\end{proof}

\begin{notation}[Realizable sign conditions]
For any finite set of polynomials $\mathcal{P} \subset \R[X_1,\ldots,X_k]$
we denote by
  \[\SIGN (\mathcal{P}) \subset \{0, 1, -
     1\}^{\mathcal{P} } \]
the set of all realizable sign conditions for
     $\mathcal{P}$  over $\R^k$, i.e.
\[
\SIGN (\mathcal{P}) = \{ \sigma \in \{0, 1, -1\}^{\mathcal{P}} \;\mid\; 
\RR(\sigma,\R^n) \neq \emptyset\}.
\]
\end{notation}

\begin{proposition}
\label{prop:BElim}
Let
$\mathcal{P} \subset \R [ Y,X]$, with
$Y = (Y_{1} , \ldots ,Y_{\ell}), X = (X_{1} ,\ldots,X_{k})$
be a finite set of polynomials. Then there exists a finite subset
$\BElim_{X} ( \mathcal{P} ) \subset \R [Y]$ 
such that 
for each 
$C \in \Cc(\BElim_{X} ( \mathcal{P} ))$,
$\SIGN (\mathcal{P} (y, X)$ 
is fixed as $y$ varies over $C$.

If the degrees of the polynomials in $\mathcal{P}$ are bounded by $d \geq 2$, then the
degrees of the polynomials in $\BElim_{X} ( \mathcal{P} )$ is bounded by
    \begin{equation}
    \label{eqn:prop:BElim}
    8 d^2 (2k(2d + 2) + 2)(2d + 3)(2d + 6)^2(2d + 5)^{2k-2} < (8 d)^{2k+4}.
    \end{equation}
\end{proposition}

\begin{proof} 
Let $\BElim_X(\mathcal{P})$ be the set of polynomials denoted by the same formula in the output of 
Algorithm 14.6 (Block Elimination) in \cite{BPRbook2posted}. The 
fact that for each 
$C \in \Cc(\BElim_{X} ( \mathcal{P} ))$,
$\SIGN (\mathcal{P} (y, X)$ 
is fixed as $y$ varies over $C$ is a consequence of 
Proposition 14.10 in \cite{BPRbook2posted}.

To obtain the upper bound on the degrees of the polynomials in 
$\BElim_{X} ( \mathcal{P})$,
we follow the complexity analysis of Algorithm 14.6 (Block Elimination) in \cite{BPRbook2posted} using
the same notation as in the algorithm. The algorithm first computes a set  $\mathrm{UR}_X(\mathcal{P})$
whose elements are tuples $v = (f,g_0,\ldots,g_k)$ of polynomials in $T,Y,\eps,\delta$ (here $\eps$ and $\delta$ are infinitesimals and $T$ is one variable).
It is proved in the complexity analysis of the algorithm 
that the degrees of the polynomials in  $T$ appearing in the various tuples
$v \in \mathrm{UR}_X(\mathcal{P})$ are bounded by
\[
D = (2d + 6)(2d + 5)^{k-2},
\]
and their degrees in $Y$ (as well as in $\eps,\delta$) are bounded by
\[
D' = (2k(2d + 2) + 2)(2d + 3)(2d + 6)(2d + 5)^{k-2}.
\]

It follows that for each $P \in \mathcal{P}$, and $v = (f,g_0,\ldots,g_k) \in \mathrm{UR}_X(\mathcal{P})$,
the degree in $T$ of the polynomial
\[
P_v = g_0^e P(\frac{g_1}{g_0},\ldots,\frac{g_k}{g_0}),
\]
where $e$ is the least even integer greater than $\deg(P) \leq d$,
is bounded by
$(d+1) D \leq 2 d D$, and similarly the degree in $Y$ of $P_v$ is bounded by $2 d D'$.
The same bounds apply to all polynomials in the set $\mathcal{F}_v$ introduced in the algorithm.

It now follows from the complexity analysis of Algorithm 11.54 (Restricted Elimination) in \cite{BPRbook2posted}, that
 the degrees in $Y$ of the polynomials in $ \mathrm{RElim}_T(f,\mathcal{F}_v)$ are bounded by
\begin{eqnarray*}
2 (2 d D) (2 d D') &=& 8 d^2 D D'\\
&=& 8 d^2 (2k(2d + 2) + 2)(2d + 3)(2d + 6)^2(2d + 5)^{2k-2} \\
&\leq & (8 d^2) \cdot ( 6 k d) \cdot  (4 d) \cdot  (5 d)^2 \cdot  (4 d)^{2k-2}\\
&=& 8\cdot 6 \cdot 4 \cdot 5^2 \cdot k \cdot d^6 \cdot  (4 d)^{2k-2}\\
&= & \frac{ 3 \cdot 5^2}{4^3} \cdot k \cdot    (4 d)^{2k+4}\\
&< & (8 d)^{2k+4}.
\end{eqnarray*}
Denoting
$\mathcal{B}_v \subset \R[Y]$ the set of coefficients of the various polynomials in 
$\mathrm{RElim}_T(f,\mathcal{F}_v)$ written as polynomials in $\eps,\delta$, we have 
immediately obtain that the degrees in $Y$ of polynomials in $\mathcal{B}_v$ are bounded by
\[
8 d^2 (2k(2d + 2) + 2)(2d + 3)(2d + 6)^2(2d + 5)^{2k-2} < (8 d)^{2k+4}.
\]

The proposition follows since according to the algorithm
\[
\BElim_X(\mathcal{P}) = \bigcup_{v \in \mathrm{UR}_X(\mathcal{P})} \mathcal{B}_v.
\]
\end{proof}

\begin{lemma}
\label{lem:BElim}
Suppose that $\mathcal{P} \subset \R [Y, X]$, 
with
$Y = (Y_{1} , \ldots ,Y_{\ell}), X = (X_{1} ,\ldots,X_{k})$
and $\Phi$ is $\mathcal{P}$-formula. 
Then there exist subsets $\mathcal{C}_\exists, \mathcal{C}_\forall \subset 
\Cc(\BElim_{X} ( \mathcal{P} ))$, such that
\begin{eqnarray}
\nonumber
\RR((\exists X) \Phi,\R^\ell) &=& \bigcup_{C \in \mathcal{C}_\exists} C, \\
\label{eqn:lem:BElim}
\RR((\forall X) \Phi,\R^\ell) &=& \bigcup_{C \in \mathcal{C}_\forall} C.
\end{eqnarray}
\end{lemma}

\begin{proof}
The lemma follows from the fact that for each 
$C \in \Cc(\BElim_{X} ( \mathcal{P} ))$,
the set
$\SIGN (\mathcal{P} (y, X)$ 
is fixed as $y$ varies over $C$ (Proposition~\ref{prop:BElim}), and the observation that 
for each $y \in \R^\ell$, the truth or falsity of each of the formulas
$(\exists X) \Phi(y,X),  (\forall X) \Phi(y,X)$
is determined by the set $\SIGN(\mathcal{P}(y,X)$.
\end{proof}

The following proposition is the key ingredient in the proof of Theorem~\ref{thm:Lojasiewicz}. It can be viewed
as a quantitative version of Proposition 2.6.4 in \cite{BCR98} (which is not quantitative). 
Our proof is similar in spirit to the proof of Proposition 2.6.4 in \cite{BCR98} but differs at certain important points making it possible to achieve the quantitative bound claimed in the proposition. 

\begin{proposition}
\label{prop:2.6.4}
Let $d \geq 2$, $\mathcal{P} \subset \R[X_1,\ldots,X_n]_{\leq d}$ and 
\[
\mathcal{P}_f, \mathcal{P}_g \subset \R[X_1,\ldots,X_n,Y]_{\leq d}
\]
be finite sets of polynomials.
Let $A \subset \R^n$ be a closed  $\mathcal{P}$-semi-algebraic set,
$f: A \rightarrow \R$ a continuous semi-algebraic function  whose
graph is a $\mathcal{P}_f$-semi-algebraic set, and 
$g: \{x \in A \;\mid\; f(x) \neq 0\} \rightarrow \R$ a continuous semi-algebraic function whose 
graph is a $\mathcal{P}_g$-semi-algebraic set.
Then there exists 
$N \leq (8 d)^{2n + 10}$ 
such that the function $f^N g$ extended by $0$ on $\{ x \in A \; \mid \;  f(x)=0\}$
is semi-algebraic and continuous on $A$. 
\end{proposition}

\begin{proof}
Suppose that $A = \RR(\Phi,\R^n)$
$\mathrm{graph}(f) = \RR(\Phi_f, \R^{n+1})$, and
$\mathrm{graph}(g) = \RR(\Phi_g, \R^{n+1})$, where
where $\Phi$ is a
$\mathcal{P}$-formula,
$\Phi_f$ is a
$\mathcal{P}_f$-formula,
and $\Phi_g$ is a
$\mathcal{P}_g$-formula.

For each $x \in A, u \in \R$, we define
\[
A_{x,u} = \{y \in A \;\mid\; ||y - x|| \leq 1, u |f(y)| = 1 \}.
\]
We define 
$\Theta(X,U,Y,V)$  to be the quantifier-free formula
$$
\displaylines{
\Phi(Y) \; \wedge \;  (||Y - X||^2 - 1 \leq 0) \cr
\;\wedge\; \cr
((V > 0 \; \wedge \; (U V  - 1 = 0)) \;\vee \; ((V < 0) \; \wedge \;  (U V +1 =0)))
\cr \;\wedge \; \cr
\phi_f(Y,V).
}
$$
Observe that for each $x\in A$ and $u \in \R$, 
\[
\RR(\Theta(x,u,\cdot,\cdot),\R^{n+1}) = \mathrm{graph}(f|_{A_{x,u}}).
\]

The semi-algebraic set $A_{x,u}$ is closed and bounded, and we define the semi-algebraic function
\[
v(x,u) = \begin{cases}
0 \mbox{ if } A_{x,u} = \emptyset, \\
\sup\{|g(y)| \;\mid\; y \in A_{x,u}\} \mbox{ otherwise}.
\end{cases}
\]

Let $\Lambda_0(X,U,W)$ denote the following the first-order (quantified) formula:

\begin{multline*}
\forall(Y,V,Z) 
(W \geq 0) \;
\wedge \;
(\Theta(X,U,Y,V) \;\wedge\;
\phi_g(Y,Z)) \;
\Longrightarrow\\
((Z \geq 0) \; \wedge \;  (W \geq Z)) \; \vee  \; ((Z \leq 0) \;  \wedge \;  (W \geq -Z))).
\end{multline*}

Finally, let $\Lambda(X,U,W)$ denote the formula
\[
(\forall W') \; \Lambda_0(X,U,W') \; \Longrightarrow \; (0 \leq W \leq W'). 
\]

Notice that for $\Lambda(x,u,w)$ is true if and only
if $w = v(x,u)$.
Also notice that for each $x\in A$, 
$\Lambda(x,U,W)$ is equivalent to a formula 
\[
\forall (Y,V,W',Z) \; \Psi_x(Y,V,Z,U,W,W')
\]
to a $(n+5)$-ary $\mathcal{P}_x$-formula $\Psi_x$, 
for some finite set  $\mathcal{P}_x \subset \R[Y,V,Z,U,W,W']_{\leq \max(d,2)}$ 
(since we assume $d \geq 2$).

Let 
\[
\mathcal{Q}_x = \BElim_{Y,V,W',Z}(\mathcal{P}_x) \subset \R[U,W]
\]
(see Proposition~\ref{prop:BElim}).

Then, using the degree bound in Proposition~\ref{prop:BElim} we have that 
for each $Q \in \mathcal{Q}_x$, $\deg(Q) < (8 d)^{2(n+3) + 4} = (8 d)^{2n +10}$.

It now follows from Lemma~\ref{lem:2dim} that there exists $c = c(x)$, such that
for all $u \geq c(x)$,
\[
|v(x,u)| \leq c \cdot u^p,
\]
with $p < (8 d)^{2n +10}$.

This means that 
\[
|f(y)|^p |g(y)| \leq c(x)
\]
on $\{y \in A \; \mid\; f(y) \neq 0 \ \mbox{ and } \ ||y - x|| \leq 1 \}$ for $|f(y)|$ sufficiently small. The function
$f^{N}g$ extended by $0$ is then semi-algebraic and continuous at $x$,
where 
$N = p+1 \leq (8 d)^{2n+10}$.
This completes the proof. 
\end{proof}

\begin{theorem}
\label{thm:2.6.6}
Let $d \geq 2$, and  
\[
\mathcal{P} \subset \R[X_1,\ldots,X_n]_{\leq d},
\mathcal{P}_f, \mathcal{P}_g \subset \R[X_1,\ldots,X_n,Y]_{\leq d}.
\]
Let $A \subset \R^n$ be a closed  $\mathcal{P}$-semi-algebraic set,
$f,g: A \rightarrow \R$ be continuous semi-algebraic functions  whose
graphs are $\mathcal{P}_f$-semi-algebraic
respectively $\mathcal{P}_g$-semi-algebraic set, and such that
$f^{-1}(0) \subset g^{-1}(0)$.
Then there exist 
$N = (8 d)^{2(n+7)}$
and a continuous semi-algebraic function
$h:A \rightarrow \R$ such that $g^N = h f$ on $A$.
\end{theorem}

\begin{proof}
Suppose that $A = \RR(\Phi,\R^n)$
$\mathrm{graph}(f) = \RR(\Phi_f, \R^{n+1})$, and
$\mathrm{graph}(g) = \RR(\Phi_g, \R^{n+1})$, where
where $\Phi$ is a
$\mathcal{P}$-formula,
$\Phi_f$ is a
$\mathcal{P}_f$-formula,
and $\Phi_g$ is a
$\mathcal{P}_g$-formula.

Let $\widetilde{A} = \{(x,f(x),g(x)) | x \in A\} \subset \R^{n+2}$.
The function $1/f$ is continuous semi-algebraic on $\{(x,u,v) \in \widetilde{A} \; \mid \;  g(x) \neq 0\}$,
and its graph is defined by the formula
\[
\Phi_g(X,V) \wedge (V \neq 0) \wedge \Phi_f(X,U) \wedge (UW - 1=0)
\]

Moreover using Proposition~\ref{prop:2.6.4} there exists $N \leq  (8 d)^{2(n+2) + 10} = (8d)^{2(n + 7)}$ such that the function 
$\widetilde{h}:\widetilde{A}  \rightarrow \R$ defined by
\[
\widetilde{h}(x,u,v) = \begin{cases}
0 \mbox{ if } f(x) = 0, \\
g^N(x)/f(x) \mbox{ if } f(x) \neq 0.
\end{cases}
\]

Since $\widetilde{h}$ does not depend on $u,v$, we get a 
continuous and semi-algebraic function $h(x) = \widetilde{h}(x,f(x),g(x))$ on $A$, and $g^N = h f$.
\end{proof}

\begin{proof}[Proof of Theorem~\ref{thm:Lojasiewicz}]
In order to prove inequality~\eqref{eqn:thm:Lojasiewicz}, 
use Theorem~\ref{thm:2.6.6} with $c = \sup_{x \in A} |h(x)|$, noticing that $c$ exists since 
$A$ is assumed to be closed and bounded.

We now prove inequality~\eqref{eqn:thm:Lojasiewicz:1}.
The set of $c \subset \mathbb{R}, c > 0$ for which inequality \eqref{eqn:thm:Lojasiewicz:0} holds for all $x \in A$ is defined by
\[
\Theta(C) := (C > 0) \wedge \left(\forall X,U,V) \left(\phi(X) \wedge \phi_f(X,U) \wedge  \phi_g(X,V) \right) \Longrightarrow
( V^{2 N} \leq C^2 \cdot U^2)\right),
\]
where $\phi$ is a $\mathcal{P}$-formula describing $A$, and $\phi_f,\phi_g$ are 
$\mathcal{Q}$-formulas describing the graphs of $f$ and $g$ 
and $N = (8 d)^{2(n+7)}$.

Using Theorem~\ref{14:the:tqe}  we obtain that 
$\Theta(C)$ is equivalent to a quantifier-free formula $\widetilde{\Theta}(C)$ such that
the bit sizes of the coefficients of the polynomials appearing in $\widetilde{\Theta}(C)$ is
bounded by $\tau d^{O(n^2)}$ and their degrees are bounded by $d^{O(n^2)}$. Now using 
Cauchy's bound (\cite[Lemma 10.2]{BPRbook2posted}), the largest real root amongst the real roots of the polynomials appearing in $\widetilde{\Theta}(C)$ is bounded by $2^{\tau d^{O(n^2)}}$. It follows that there exists $c = 2^{\tau d^{O(n^2)}}$ for which the inequality~\eqref{eqn:thm:Lojasiewicz:0} holds.

Using the repeated squaring technique (see below) at the cost of introducing $O(n \log d)$ new variables, it is possible write another universally quantified formula, namely
\begin{multline*}
\Theta'(C) := (C > 0) \wedge (\forall T_1,\ldots,T_M,,X,U,V) (\phi(X) \wedge \phi_f(X,U) \wedge \phi_g(X,V) \wedge \cr
(T_1= V) \wedge (T_2 = T_1^2) \wedge \cdots \wedge (T_{M} = T_{M-1}^2)) \Longrightarrow
( T_M^2 \leq C^2 \cdot U^2)),
\end{multline*}
equivalent to $\Theta(C)$ in which all the polynomials appearing have degrees
bounded by $d$ (instead of $d^{O(n)}$ as in the formula $\Theta$). 
The number of quantifier variables in the formula $\Theta'$ equals $M+n+2$, where 
$M = O(\log N) = O(n \log d)$.

Now using Theorem~\ref{14:the:tqe} and Cauchy's bound as before we obtain a bound of
$2^{\tau d^{O(n\log d)}}$ on $c$.
\end{proof}

\hide{
\begin{remark}
Theorem~\ref{thm:Lojasiewicz} in particular improves the {\L}ojasiewicz exponent in~\eqref{ineq:Kurdyka_bound}. To see this, let us define the semi-algebraic functions $f,g: S \to \mathbb{R}$ defined by $f(x) = \|F(x)\|$ and $g(x)= \distance\big(x,F^{-1}(0) \cap S\big)$, where $f^{-1}(0) \subset g^{-1}(0)$. Then Theorem~\ref{thm:Lojasiewicz} yields the {\L}ojasiewicz exponent $d^{O(n^2)}$, where $d$ is an upper bound on the degree of polynomials in $F$ and the degree of polynomials describing the semi-algebraic set $S$. Notice that the upper bound does not include the number of polynomial equations.
\end{remark}}

\subsection{Proof of Theorem~\ref{thm:error_bound_semialg_set}}
First, we need the following lemma.

\begin{lemma}
\label{upper_bound_dist_func}
Let $\mathcal{P} \subset \R[X_1,\ldots,X_n]_{\le d}$ and $S \subset \R^n$ 
a $\mathcal{P}$-semi-algebraic set. Then there
exists $\mathcal{Q} \subset \R[X_1,\ldots,X_n,U]$ such that the graph of the function 
$\distance(\cdot,S):\R^n \rightarrow \R$ is a 
$\mathcal{Q}$-semi-algebraic set and $\max_{Q \in \mathcal{Q}} \deg(Q) = d^{O(n)}$. 
\end{lemma}

Before proving Lemma~\ref{upper_bound_dist_func} we need a new notion (that of Thom encoding of real roots 
of a polynomial) that will be needed in the proof.
The following proposition appears in \cite[Proposition~2.36]{BPRbook2posted}.

\begin{proposition} [Thom's Lemma] \cite[Proposition~2.36]{BPRbook2posted} 
\label{prop:Thom}
Let $P \subset \R[X]$ be a univariate polynomial, $\Der(P)$ the set of derivatives of $P$
  and $\sigma \in \{-1,0,1\}^{\Der(P)}$. Then $\RR ( \sigma ,\R)$ is either empty, a
  point, or an open interval.
\end{proposition}

Note that it follows immediately from Proposition~\ref{prop:Thom}, that for any
$P \in \R[X]$ and $x \in \R$ such that $P(x) = 0$, the sign condition $\sigma$ realized by
$\Der(P)$ at $x$ characterizes the root $x$. We call $\sigma$ the \emph{Thom encoding} of the root $x$
of $P$.

\begin{proof}[Proof of Lemma~\ref{upper_bound_dist_func}]
Let $\Phi$ be a $\mathcal{P}$-formula such that $\RR(\Phi,\R^n) = S$.
Let
\[
W = \{(x,t) \;\mid\; \exists y \in S \mbox{ with } t = ||x-y||\}.
\]

Then for each $x \in \R^n$, 
\[
\distance(x,S) = \inf W_x,
\]
where $W_x$ denotes the one-dimensional fiber of $W$ over $x$ with respect to the projection 
along the $t$ coordinate.

It is also clear from the definition that $W$ is the image under projection along the $y$  coordinates) of the semi-algebraic set
$V \subset \R^n \times \R^n \times \R$
defined by 
\[
V = \{(x,y,t) \;\mid\; y \in S \mbox{ and  } t = ||x-y||\}.
\]
Let $\Theta(X,Y,T)$ be the formula
\[
\Phi(Y) \wedge (T^2 = ||X-Y||^2) \wedge (T \geq 0).
\]
Then, it is clear that $\Theta$ is a $\mathcal{P}'$-formula for some finite subset 
$\mathcal{P}' \subset \R[X,Y,T]_{\leq d}$ (assuming $d \geq 2$), and moreover
$\RR(\Theta, \R^{2n+1}) = V$.

Now apply Theorem~\ref{14:the:tqe} to the quantified formula
$(\exists Y) \Theta(X,Y,T)$ and obtain a quantifier-free $\mathcal{F}$-formula,
$\widetilde{\Theta}$ equivalent to $\Theta$, where 
$\mathcal{F}$ is some finite subset of $\R[X,T]$ and 
\[
D = \max_{F \in \mathcal{F}} \deg(F) = d^{O(n)}.
\]

Notice that $\RR(\widetilde{\Theta},\R^{n+1}) = W$, and for each $x \in \R^n$,
\[
\inf W_x = \distance(x,S)
\]
is a real root of some polynomial in $\mathcal{F}$.

Let $\mathcal{F} = \{F_1,\ldots,F_N\}$. 
We denote 
\[
\Der_T(F_i) = \{F_i, F_i',\ldots,F_i^{(D)} \}
\]
the set of derivatives of $F_i$ with respect to $T$.

For $1 \leq i \leq N$, and $\sigma \in \{-1,0,1\}^{\Der_T(F_i)}$, denote by $\Psi_{i,\sigma}$ the quantifier-free formula,
\[
(\forall T) \left(\bigwedge_{0 \leq j \leq D} (\sign(F_i^{(j)}) = \sigma(F_i^{(j)}))\right) \Longrightarrow \left((\forall T')
\widetilde{\Theta}(X,T') \Longrightarrow (T \leq T')\right).
\]

Using Theorem~\ref{14:the:tqe} one more time, let $\widetilde{\Psi}_{i,\sigma}$ be a quantifier-free formula
equivalent to $\Psi_{i,\sigma}$ and let the set of polynomials appearing in $\widetilde{\Psi}_{i,\sigma}$ be denoted by
$\mathcal{Q}_{i,\sigma}$. 

The semi-algebraic set $\RR(\widetilde{\Psi}_{i,\sigma},\R^n)$ is the set consisting of points $x \in \R^n$, such that the $\distance(x,S)$ equals the (at most one) real root of the polynomial $F_i(x,T)$ with Thom encoding $\sigma$. Notice that the maximum degree of polynomials in $\mathcal{Q}_{i,\sigma}$
is bounded by $D^{O(1)} = d^{O(n)}$.

Finally, let 
\[
\Psi = \bigwedge_{i,\sigma} \left(\widetilde{\Psi}_{i,\sigma} \Longrightarrow \left(\bigwedge_{0 \leq j \leq D} (\sign(F_i^{(j)}) = \sigma(F_i^{(j)}))\right)\right),
\]
and 
\[
\mathcal{Q} = \bigcup_{1 \leq i \leq N} \left(\Der_T(F_i) \cup \bigcup_{\sigma \in \{-1,0,1\}^{\Der_T(F_i)}} \mathcal{Q}_{i,\sigma}\right).
\]
It is clear from the above construction that, $\Psi$ is a $\mathcal{Q}$-formula, and
$\RR(\Psi,\R^{n+1})$ is the graph of the function $\distance(\cdot,S)$, and the degrees of the polynomials in
$\mathcal{Q}$ are bounded by $d^{O(n)}$. This proves the lemma.
\end{proof}

\begin{remark}
In~\cite[Proposition~2.2.8]{BCR98} and~\cite[Theorem~7]{S91}, the graph of the semi-algebraic function $\distance(x,\mathcal{M})$ is described by the following quantified formula with two blocks of quantifiers
\begin{equation}\label{distance_formula}
    (\exists Y) \ (\forall Y') \ \ (\neg \Psi(Y') \vee (\Psi(Y) \wedge (||X-Y||^2 = T^2) \wedge (T \ge 0) \wedge (||X-Y'||^2\ge T^2)),  
\end{equation}
where $(\exists Y)$ and $(\forall Y')$ are
two blocks of variables each of size $n$ with different quantifiers. The effective quantifier elimination 
(Theorem~\ref{14:the:tqe})
applied to~\eqref{distance_formula} yields a quantifier-free formula with polynomials having degrees bounded by  $d^{O(n^2)}$, where $d$ is an upper bound on the  degrees of the polynomials in $\Psi$. 
The formula given in Lemma~\ref{upper_bound_dist_func} describing the graph of the same
function involves polynomials of degrees at most $d^{O(n)}$ (though it may involve many more polynomials
than the one in \eqref{distance_formula}) 
which is an improvement over the bound of $d^{O(n^2)}$ obtained by the above argument. Note that for our purposes, the degrees of the polynomials
appearing in the formula is the important quantity -- the number of polynomials appearing is not relevant.
\end{remark}

\begin{proof}[Proof of Theorem~\ref{thm:error_bound_semialg_set}]
It is easy to see that the residual function $\psi(x)$ defined in~\eqref{func:residual_function} satisfies 
\begin{align*}
    \distance(x,M) = 0 \ \ \Longleftrightarrow \ \ \psi(x) = \sum_{j=1}^{s} |h_j(x)| + \sum_{i=1}^{r} \max\{g_i(x),0\} = 0.
\end{align*}

\vspace{5px}
\noindent
The graph of $\psi(x)$ can be described using a quantifier-free $\mathcal{P}$-formula with $\mathcal{P} \subset \R[X_1,\ldots,X_n]_{\le d}$ as follows (note that we do not care about the cardinality of $\mathcal{P}$).

To see this first observe that for all $x \in \R^n$, and $\sigma \in \{0,1,-1\}^{[1,s]}$
if $\sign(h_j(x)) = \sigma(j), j \in [1,s]$, then
\[
\sum_{j=1}^{s} |h_j(x)| = \sum_{j=1}^s \sigma(j) h_j(x).
\]
Similarly, 
for all $x \in \R^n$, and $\tau \in \{0,1,-1\}^{[1,r]}$
if $\sign(g_i(x)) = \tau(i), i \in [1,r]$, then
\[
\sum_{i=1}^{r} \max\{g_i(x),0\} = \frac{1}{2}\sum_{i=1}^{r} \tau(i)(1 + \tau(i)) g_i(x).
\]

It is now easy to verify that the following quantifier-free formula in $n+1$ variables:

\begin{multline}
\bigvee_{
\substack{
\sigma \in \{0,1,-1\}^{[1,s]} \\
\tau \in \{0,1,-1\}^{[1,r]}
}
}
\left(\bigwedge_{j=1}^{s}(\sign(h_j) = \sigma(j)) \wedge \bigwedge_{i=1}^{r}(\sign(g_i) = \tau(i))
\right)
\Longrightarrow \\
\left(T - \sum_{j=1}^s \sigma(j) h_j - \frac{1}{2}\sum_{i=1}^{r} \tau(i)(1 + \tau(i)) g_i=0
\right)
\end{multline}
describes the graph of $\psi$ and all polynomials occurring in it have degrees at most $d$.

Moreover, it follows from Lemma~\ref{upper_bound_dist_func} that 
the graph of $\distance(\cdot,M)$ can be defined by a quantifier-free formula involving 
polynomials in $n+1$ variables having degrees bounded by $d^{O(n)}$.
The first part of the theorem now follows from Theorem~\ref{thm:Lojasiewicz} after setting $f(x)= \psi(x)$, and $g(x)=\distance(x,M)$. 

\vspace{5px}
\noindent
In case that the $M$ is finite, it is possible to derive a sharper estimate of the error bound exponent. Suppose $M = \{p_1, \ldots, p_N \} \subset \R^n$.
In this case the graph of the distance function,
$\distance(X,M)$,
is described by the following formula

\[
\Theta(X,T) :=  (T \geq 0) \wedge \left(\bigwedge_{i=1}^N (T^2 \geq ||X - p_i||^2)\right) \wedge 
\left (\bigvee_{i=1}^{N} (T^2 = ||X - p_i||^2)\right).
\]

Notice that the degrees of the polynomials appearing in the quantifier-free formula $\Theta$ are
bounded by $2$. Furthermore, the graph of the residual function $\psi$ is defined by quantifier-free formula involving polynomials of degrees bounded by $d$. The second part of the theorem is now immediate from Theorem~\ref{thm:Lojasiewicz}.
\end{proof}

\begin{remark}
A weaker version of Theorem~\ref{thm:error_bound_semialg_set} can be directly obtained using the quantifier elimination. Let us define 
\begin{align*}
M_{\varepsilon}:=\{x \in \R^n \mid  \psi(x) = \varepsilon, \ x \in E\}.
\end{align*}
Then it suffices to describe the graph of the semi-algebraic function
\begin{align*}
\varphi: \varepsilon \mapsto \sup_{x \in M_{\varepsilon}} \distance(x,M).
\end{align*}
Let $\Phi(X,\varepsilon)$ be a quantifier-free formula describing $M_{\varepsilon}$ and $\Psi(X,T)$ be the quantifier-free formula describing the graph of $\distance(x,M)$ in the proof of Lemma~\ref{upper_bound_dist_func}. Then 
the graph of 
$\varphi$ is the lower envelope of the realization of the formula~\eqref{lower_envelope_formula}:
\begin{align}\label{lower_envelope_formula}
 (\forall X) \ (\exists T) \quad \neg (\Phi(X,\varepsilon) \wedge \Psi(X,T)) \vee (\varphi \ge T).  
\end{align}
By Lemma~\ref{upper_bound_dist_func} and the application of Theorem~\ref{14:the:tqe} to~\eqref{lower_envelope_formula}, the graph of $\varphi$ can be described by a quantifier-free $\mathcal{P}$-formula with $\mathcal{P} \subset \R[\varphi,\varepsilon]_{\le d^{O(n^2)}}$, which implies the existence of $P \in \mathcal{P}$ such that $P(\varphi,\varepsilon)= 0 $ for sufficiently small positive $\varepsilon$. By the Newton-Puiseux theorem~\cite[Theorems 3.2 and~4.1 of Chapter IV]{W78}, $\psi$ near $\varepsilon = 0$ can be expanded as
\begin{align*}
    \varphi(\varepsilon) = a_1 \varepsilon^{\gamma_1} + a_2 \varepsilon^{\gamma_2} + \ldots 
\end{align*}
where $0< \gamma_1 < \gamma_2 < \ldots$, and $a_i \in \R$. Furthermore, it follows from the Newton-Puiseux algorithm~\cite[Section~3.2 of Chapter IV]{W78} that $\gamma_1 \le 1/\deg(P)$,  where $\deg(P) = d^{O(n^2)}$. All this implies $\varphi^{\deg(P)} = O(\varepsilon)$ for sufficiently small positive $\varepsilon$, and thus 
\
\begin{align*}
   \distance(x,M)^{\deg(P)} \le \kappa \cdot \psi(x), \quad \mbox{for all } x \in E \mbox{ with sufficiently small }  \psi(x). 
\end{align*}
The case with $\dim M=0$ is analogous.
\end{remark}

\begin{proof}[Proof of Corollary~\ref{cor:semi-definite_sys}]
Notice that $M \cap \mathbb{S}^p_+$ can be redefined as a basic $\mathcal{Q}$-semi-algebraic set with $\mathcal{Q} \subset \R[X_1,\ldots,X_{p^2}]_{\le \max\{d,p\}}$ and $|\mathcal{Q}|=2^p+r-1$ as follows:
\begin{align*}
\big\{X \in \mathbb{S}^p \mid g_i(X) \le 0, \ i=1,\ldots,r, \  \det(X_{I}) \ge 0, \ \forall I \subseteq \{1,\ldots,p\}\big\},    
\end{align*}
where $X_{I}$ is a principal submatrix of $X$ indexed by $I$. We also define the residual function
\begin{align*}
    \psi(x):=\max\Big\{\distance(x,M),\max_{I \subseteq \{1,\ldots,p\}} \big(\max\{-\det(X_{I}),0\}\big)\Big\},
\end{align*}
which is a $\mathcal{Q}'$-semi-algebraic function with $\mathcal{Q}' \subset \R[X_1,\ldots,X_{p^2},Y]_{\le d^{O(p^2)}}$, see Lemma~\ref{upper_bound_dist_func} and the proof of Theorem~\ref{thm:error_bound_semialg_set}. Then by applying Theorem~\ref{thm:error_bound_semialg_set}, Lemma~\ref{upper_bound_dist_func}, and  Remark~\ref{rem:arbitrary_residual_functions} to $M \cap \mathbb{S}^p_+$, $E$, and $\psi(x)$ we get 
\begin{multline*}
\distance(x,M \cap \mathbb{S}_+^p)^{\rho} \le \kappa' \cdot \max\Big\{\distance(x,M),\max_{I \subseteq \{1,\ldots,p\}} \big(\max\{-\det(X_{I}),0\}\big)\Big\}, \\ \mbox{ for all } x \in E,
\end{multline*}
for some $\kappa' > 0$ and $\rho=\max\{d,p\}^{O(p^4)}$. The rest of the proof follows from the arguments in~\cite{LP98}. Let $\lambda_i(X_{I})$ be the $i$th smallest eigenvalue of $X_{I}$. By the boundedness of $E$, there exists some positive $r$ such that $|\lambda_i(X_{I})| \le r$ for all $i=1,\ldots,|I|$. Furthermore, by the interlacing property of eigenvalues of $X$, we have
\begin{align*}
\lambda_{\min}(X_{I}):=\lambda_{1}(X_{I}) \ge \lambda_{\min}(X),
\end{align*}
which yields 
\begin{align*}
    \det(X_{I}) = \prod_{i=1}^{|I|} \lambda_{i}(X_{I}) \ge r^{|I|-1} \lambda_{\min}(X_{I}) \ge  r^{|I|-1} \lambda_{\min}(X).  
\end{align*}
Consequently,
\begin{align*}
 \max\{-\det(X_{I}),0\} \le r^{|I|-1} \cdot \max\{-\lambda_{\min}(X),0\}   
\end{align*}
for every $I \subseteq \{1,\ldots,p\}$, which completes the proof.  
\end{proof}

\subsection{Proof of Theorem~\ref{thm:uniform}}
In the proof of Theorem~\ref{thm:uniform} we need the following ingredient which is proved in 
\cite[Theorem 2.5]{Basu9}.

\begin{proposition}[Quantitative cylindrical definable cell decomposition]
\label{prop:qcad}
Fix an o-minimal expansion of the real closed field $\R$.
Let $\mathcal{A}$ be a definable family of subsets of $\R^n$ parametrized by the definable set $A$. Then there
exist a finite set $J$ and definable families $(\mathcal{C}_j)_{j \in J}$ of subsets of $\R^n$ each parametrized by
$A^{N(n)}$ where $N(n) = 2(2^n-1)$ having the following property.
Suppose, $A' \subset A$ is a finite subset. Then, there exists a cylindrical decomposition
$\mathcal{D} = (\mathcal{D}_1,\ldots,\mathcal{D}_n)$ of $\R^n$,
such that for each $i, 1 \leq i \leq n$,
the set of cells of the cylindrical decomposition $(\mathcal{D}_1,\ldots,\mathcal{D}_i)$ of $\R^i$ is a subset of the set of definable sets
\[
\{ (\mathcal{C}_{j})_{\overline{a}} \; \mid \; \overline{a} \in A'^{N(n)} \}.
\] 
\end{proposition}

For the rest of this section we fix a polynomially bounded o-minimal expansion of $\mathbb{R}$.

\begin{proposition}[o-minimal quantitative version of Proposition 2.6.4 in \cite{BCR98}]
\label{prop:2.6.4:o}
Let $\mathcal{A}$ be a definable family of subsets of $\mathbb{R}^n$ parametrized by the definable set $A$,
and let $\mathcal{B}$ be a definable subset of $\mathbb{R}^{n+1}$ parametrized by the definable set $B$.

Then there exists $N = N(\mathcal{A},\mathcal{B}) > 0$ having the following property.
For any triple of finite sets
$(A',B',B'')$ with $A' \subset A, B',B'' \subset B$, 
a closed $(\mathcal{A},A')$-set $S$, 
a $(\mathcal{B},B')$-set $F$, $(\mathcal{B},B'')$-set $G$, such that $F,G$ are graphs of
definable functions $f:\mathbb{R}^n \rightarrow \mathbb{R}, g: \mathbb{R}^n \rightarrow \mathbb{R}$,
such that $f|_S$ and $g|_{S_{f\neq 0}}$
(where $S_{f \neq 0} = \{x \in S \;\mid\; f(x) \neq 0\}$)
are continuous,
the function $f^N g|_{S_{f\neq 0}}$ extended by $0$ on $\{ x \in S \; \mid \;  f(x)=0\}$
is continuous on $S$. 
\end{proposition}

\begin{proof}
We follow closely the proof of the corresponding proposition (Proposition~\ref{prop:2.6.4}) in the semi-algebraic case.

For each $x \in S, u \in \mathbb{R}$, we define
\[
S_{x,u} = \{y \in S \;\mid\; ||y - x|| \leq 1, u |f(y)| = 1 \}.
\]

The set $S_{x,u}$ is definable, closed and bounded, and we define the function
\[
v(x,u) = \begin{cases}
0 \mbox{ if } A_{x,u} = \emptyset, \\
\sup\{|g(y)| \;\mid\; y \in A_{x,u}\} \mbox{ otherwise}.
\end{cases}
\]

Clearly, $v: S \times \mathbb{R} \rightarrow \mathbb{R}$ is a definable function.

We define 
$$
\displaylines{
T = \{(x,u,y,v) \mid y \in S \wedge  
(||y - x||^2 - 1 \leq 0) \cr
\;\wedge\; \cr
((v > 0 \; \wedge \; (u v  - 1 = 0)) \;\vee \; ((v < 0) \; \wedge \;  (u  v +1 =0)))
\cr \;\wedge \; \cr
(y,v) \in F \}.
}
$$

Notice that $T$ is definable and for each fixed $x$, $T_x \subset \R^{n+2}$ is a $(\mathcal{C}, A' \times B')$-set where
$\mathcal{C}$ is a definable family of subsets of $\R^{n+2}$ parametrized by 
$(A \times B)$, and $\mathcal{C}$ depends only on the definable families $\mathcal{A},\mathcal{B}$.

Observe  also that for each $x\in A$ and $u \in \R$, 
\[
T_{x,u}  = \mathrm{graph}(f|_{S_{x,u}}).
\]

We also define the set $L_0$ by
$$
\displaylines{
L_0 = \{(x,u,w) \; \mid \;
\forall(y,v,z) \cr
(w \geq 0)
\cr
\wedge
\cr
((x,u,y,v) \in T \wedge
(y,z) \in G)
\Longrightarrow \cr
((z \geq 0) \wedge (w\geq z)) \vee ((z \leq 0) \wedge (w \geq -z)))
\}
.
}
$$
Finally, let
\[
L = \{(x,u,w) \; \mid \; (\forall w') 
(x,u,w') \in L_0 \Longrightarrow (0 \leq w \leq w') \}.
\]
Notice that for $x \in S$,  $(x,u,w) \in L$  if and only
if $w = v(x,u)$.

Also notice that
the set $L \subset S \times \mathbb{R}^2$ 
is a definable set which is the complement of a projection of
an $(\mathcal{D},D')$-set $P \subset \mathbb{R}^{2n+5}$,
where $\mathcal{D}$ is a definable family of sets parametrized by $A \times A \times B \times B$
and $D' = A' \times A' \times B' \times B''$, and $\mathcal{D}$ depends only on the definable
families $\mathcal{A},\mathcal{B}$.

We now apply Proposition~\ref{prop:qcad} to the definable family $\mathcal{D}$. There 
exist a set of definable families $(\mathcal{C}_j)_{j \in J}$ and 
a cylindrical decomposition $(\mathcal{D}_1,\ldots,\mathcal{D}_{2n+5})$ of $\mathbb{R}^{2n+5}$ adapted to the set $P$
whose cells are of the form
$(\mathcal{C}_{j})_w$ with $j \in J, w \in (D')^{N(2n+5)}$.

For $x \in S$, there exists a unique cell,
$C = (\mathcal{C}_{j})_{w}$, $j \in J, w \in (D')^{N(2n+5)}$, 
of the decomposition $\mathbb{R}^n$ containing $x$.

The definable functions $v(x,\cdot)$ as 
$x$ varies over $(\mathcal{C}_{j})_{w}$ and $w$ varies over $(A \times A \times B \times B)^{N(2n+5)}$ and $j \in J$
form a  definable family, and using Proposition 5.2 in \cite{Miller} 
there exists $p = p(\mathcal{A},\mathcal{B})$
such that 
\[
|v(x,u)| \leq c(x) \cdot u^p, 
\]
for all $u$ large enough.

Now for each  $x \in S$, there exists a cell,
$C = (\mathcal{C}_{j})_{w}$ $j \in J, w \in (D')^{N(2n+5)}$, 
of the decomposition $\mathbb{R}^n$ containing $x$.

This proves the proposition taking $N = p+1$.
\end{proof}

\begin{theorem}
\label{thm:2.6.6-o}
Let $\mathcal{A}$ be a definable family of subsets of $\R^n$ parametrized by the definable set $A$,
and let $\mathcal{B}$ be a definable subset of $\R^{n+1}$ parametrized by the definable set $B$.
Then there exists $N = N(\mathcal{A},\mathcal{B}) > 0$ having the following property.
For any triple of finite sets
$(A',B',B'')$ with $A' \subset A, B',B'' \subset B$, 
a closed $(\mathcal{A},A')$-set $S$, 
a $(\mathcal{B},B')$-set $F$, $(\mathcal{B},B'')$-set $G$, such that $F,G$ are graphs of
definable functions $f:\mathbb{R}^n \rightarrow \mathbb{R}, g: \mathbb{R}^n \rightarrow \mathbb{R}$ continuous on $S$, such that
$f|_S^{-1}(0) \subset g|_S^{-1}(0)$.
Then there exists a continuous definable function
$h:S \rightarrow \R$ such that $g|_S^N = h f|_S$ on $S$.
\end{theorem}

\begin{proof}
Similar to the proof of Theorem~\ref{thm:2.6.6} replacing semi-algebraic by definable.
\end{proof}

\begin{proof}[Proof of Theorem~\ref{thm:uniform}]
Use Theorem~\ref{thm:2.6.6-o} with $c = \sup_{x \in S} |h(x)|$.
\end{proof}

\section{Applications to optimization}
\label{sec:applications}
As an illustration of the improvement one obtains by applying the improved bound on the 
{\L}ojasiewicz exponent proved in Theorem~\ref{thm:Lojasiewicz} and the error bound  in Theorem~\ref{thm:error_bound_semialg_set}
we consider the following application in the theory of optimization. Clearly, Theorem~\ref{thm:error_bound_semialg_set} can be applied to other situations as well where error bounds are important, 
for example in the study of  H\"olderian continuity of the set-valued map defined  by~\eqref{eq:basic_semi_error_bound} as stated in~\cite[Theorem~3.1]{LP94}. 

\subsection{Binary feasibility problems}
We can use Theorem~\ref{thm:error_bound_semialg_set} and its independence from the number of constraints to derive an error bound for a binary feasibility problem, where the feasible set is defined by
\begin{align}\label{binary_set}
M=\{x \in \R^n \mid g_i(x) \le 0, \ h_j(x) &= 0, \ i=1,\ldots,r, \ j=1,\ldots,s, \\ 
x_k &\in \{0,1\}, \ k=1,\ldots,n\}, \nonumber
\end{align}
where $g_i,h_j \in \R[X_1,\ldots, X_n]_{\le d}$. The following result is a quantified version of~\cite[Theorem~5.6]{LP94} specialized for polynomials.
\begin{corollary}
Let $M$ be defined in~\eqref{binary_set} and let $E$ be a closed and bounded $\mathcal{P}$-semi-algebraic subset of $\R^{n}$ 
with $\mathcal{P} \subset \R[X_1,\ldots,X_{n}]_{\leq d}$. Then there exist $\kappa >0$ and 
$\rho = (O(d))^{2n+14}$
such that 
\begin{align*}
    \distance(x,M)^{\rho} \le \kappa \cdot \psi (x), \quad \mbox{for all }  x \in E,
\end{align*}
where
\begin{align*}
    \psi(x) = \sqrt{\sum_{j=1}^{s} (h_j(x))^2} + \sqrt{\sum_{i=1}^{r} (\max\{g_i(x),0\})^2} +\sum_{k=1}^n \abs{x_k(1-x_k)}.
\end{align*}
\end{corollary}

\begin{proof}
Note that for every $k=1,\ldots,n$, the binary constraint on $x_{k}$ can be enforced by
\begin{align*}
x_{k}(1-x_{k}) = 0, \quad  x_{k} \in [0,1]\subset\R.    
\end{align*}
Then $M$ can be redefined as  
\begin{align}\label{redefined_discrete_set}
M  :=\Big\{x \in \R^n \mid g_i(x) \le 0, \ h_j(x) &= 0, &  i&=1,\ldots,r, \ j=1,\ldots,s,\\ \nonumber
x_{k}(1-x_{k}) &= 0, & k&=1,\ldots,n,\Big\},
\end{align}
which is a finite subset of $\R^n$. 

Also observe that defining 
\begin{align*}
     \widetilde{\psi}(x)  = \sum_{j=1}^{s} |h_j(x)| + \sum_{i=1}^{r} \max\{g_i(x),0\} +\sum_{k=1}^n \abs{x_k(1-x_k)},   
\end{align*}
it follows from the Cauchy-Schwarz inequality 
that 
\[
\widetilde{\psi}(x) \leq c \cdot \psi(x),
 \]
with $c = \max\{1,\sqrt{r},\sqrt{s}\}> 0$.

Now the result follows by applying 
the second part of 
Theorem~\ref{thm:error_bound_semialg_set} and Remark~\ref{rem:arbitrary_residual_functions} to~\eqref{redefined_discrete_set} and the residual function
$\widetilde{\psi}(x)$.
\end{proof}

\subsection{Convergence rate of feasible descent schemes}
 Error bounds are important to estimate the convergence rate of iterative algorithms in nonlinear optimization. Here, we present the convergence analysis in~\cite[Theorem~5]{L2000}, which is relevant to Theorem~\ref{thm:Lojasiewicz}. 
 
 Let $\R=\mathbb{R}$ and $g_i,h_j$ defined in~\eqref{eq:basic_semi_error_bound} be convex polynomials. Then the goal of a feasible descent scheme is to find \textit{stationary solutions} of a polynomial $f \in \mathbb{R}[X_1,\ldots,X_n]_{\le d}$ over a convex $M$ (assuming that $\inf_{x \in M} f > -\infty$), where a stationary solution is an $x \in \mathbb{R}^n$ such that
 \begin{align*}
     x - \proj_M\!\big(x-\nabla f(x)\big) = 0,
 \end{align*}
 in which $\proj_M(\cdot)$ denotes the projection onto the convex set $M$. The idea of a feasible descent scheme is to generate a sequence $\{x_k\}_{k=1}^{\infty}$ of solutions by
 \begin{align}\label{feasible_descent_scheme}
     x_{k+1} = \proj_M\!\big(x_{k} - \alpha_k \nabla f(x_k) + e_k\big),
 \end{align}
 where $\alpha_k > 0$ is so-called the step length and $e_k$ is an error vector depending on $x_k$. If we define the set of stationary solutions as
  \begin{align*}
     S:=\Big\{x \in \mathbb{R}^n \mid x - \proj_M\!\big(x-\nabla f(x)\big) = 0\Big\},
 \end{align*}
 then the following result is well-known for the convergence rate of a feasible descent scheme which we specialize for the polynomial $f$.
 \begin{proposition}[Theorem~5 in~\cite{L2000}]\label{feasible_descent_convergence}
 Suppose that the gradient of $f$ is Lipschitz continuous and there exists $\varepsilon > 0$ such that
 \begin{align*}
     x \in S, \ \ y \in S, \ \ f(x) \neq f(y) \Longrightarrow \|x-y\| \ge \varepsilon.
 \end{align*}
 If $\liminf \alpha_k > 0$ and the sequences $\{e_k\}_{k=1}^{\infty}$ and $\{x_k\}_{k=1}^{\infty}$ generated by~\eqref{feasible_descent_scheme} satisfy
 \begin{align*}
     \|e_k\| \le \kappa_1 \|x_k - x_{k+1}\|, \ \text{for some} \ \kappa_1 > 0,\\
     f(x_{k+1}) - f(x_k) \le -\kappa_2 \|x_k - x_{k+1}\|^2, \ \text{for some} \ \kappa_2 > 0,
 \end{align*}
 then the sequence $\{f(x_k)\}_{k=1}^{\infty}$ converges at least sub-linearly, at the rate $k^{1-\rho}$, where $\rho \ge 1$ is an integer satisfying
 \begin{align*}
     \distance(x,S)^{\rho} \le \kappa \cdot \|x - \proj_M\!\big(x-\nabla f(x)\big)\|
 \end{align*}
 for some $\kappa > 0$ and for all $x$ in a compact semi-algebraic subset of $M$.
 \end{proposition}
 Notice that $S$ and $\|x - \proj_M\!\big(x-\nabla f(x)\big)\|$ are both semi-algebraic, and the latter is a residual function. Therefore, Theorem~\ref{thm:Lojasiewicz} and Lemma~\ref{upper_bound_dist_func} can be applied to quantify the convergence rate $k^{1-\rho}$ in terms of $d$ and $n$ only.

\begin{remark}
The upper bound~\eqref{upper_bound_convex_semi-algebraic} was used to quantify the convergence rate of the cyclic projection algorithm applied to finite intersections of convex semi-algebraic subsets of $\mathbb{R}^n$ ~\cite[Theorem~4.4]{BLY14}.
\end{remark}

\subsection{Sums of squares relaxation}
A polynomial optimization problem is formally defined as the following:
given a polynomial $f \in \mathbb{R}[X_1,\ldots, X_n]_{\le d}$,
compute 
 \begin{align}\label{poly_optim_def}
 f_{\min}^*:=\inf \big \{f(x) \mid x \in M \big\},
 \end{align}
where $M$ is defined in~\eqref{eq:basic_semi_error_bound} with $\R=\mathbb{R}$. We assume, without loss of generality, here that $s = 0$ in~\eqref{eq:basic_semi_error_bound}.

Unlike semi-definite optimization, see e.g.,~\cite{BM22}, there is no efficient interior point method for polynomial optimization. Nevertheless, tools in real algebra have laid the groundwork for developing an efficient numerical approach, where semi-definite relaxation plays a central role. Using this approach,~\eqref{poly_optim_def} is approximated by a hierarchy of semi-definite relaxations, so-called \textit{sums of squares (SOS)} relaxations~\cite{L01,P03}, see also~\cite{L09}. A SOS relaxation of order $t$ is defined as 
\begin{align}\label{SOS_relaxation}
f_{\mathrm{sos}}^{*t}=\sup\{\beta \mid f-\beta \in \mathbf{\mathcal{M}}_{2t}(g_1,\ldots,g_r)\},
\end{align}
where
\begin{align*}
\mathbf{\mathcal{M}}_{2t}(g_1,\ldots,g_r)&=\\
&\bigg\{u_0 - \sum_{j=1}^r u_jg_j \mid u_0,u_j \in \Sigma, \ \deg(u_0), \deg(u_jg_j) \le 2t, \ j=1,\ldots,r\bigg\}
\end{align*}
is called the 
\textit{truncated quadratic module generated by $g_1,\ldots,g_r$} and $\Sigma$ is the convex cone of SOS polynomials. It is worth noting that~\eqref{SOS_relaxation} is a semi-definite optimization problem of size $O(n^t)$. Under some conditions on $M$, see e.g.,~\cite[Proposition~6.2 and Theorem~6.8]{L09},~\eqref{SOS_relaxation} is feasible for sufficiently large $t$, and $f_{\mathrm{sos}}^{*t} \to f_{\min}^*$ as $t \to \infty$.  

\vspace{5px}
\noindent
Recently, Baldi and Mourrain~\cite{BMo22} provided a convergence rate for $f_{\mathrm{sos}}^{*t}$ in terms of $t$, in which the {\L}ojasiewicz exponent plays a central role. 
The authors proved~\cite[Theorem~4.3]{BMo22}, under some conditions, that there exists $c > 1$ such that 
\begin{align*}
0 < f_{\min}^* -  f_{\mathrm{sos}}^{*t} \le c \cdot \|f\| \deg(f)^{\frac 75} t^{-\frac{1}{2.5n\rho}},
\end{align*}
where $c$ depends on $n$, $\rho$, and $g_1,\ldots, g_r$,  $\|f\|:=\max_{x \in [-1,1]^n} |f(x)|$, and $\rho$ is the error bound exponent for the inequality
\begin{align*}
    \distance(x,M)^{\rho} \le \kappa \cdot |\min\{g_1(x),\ldots,g_r(x),0\}|, \quad \text{for all} \ x\in [0,1]^n
\end{align*}
for some $\kappa > 0$. Then the application of Theorem~\ref{thm:error_bound_semialg_set} and Remark~\ref{rem:arbitrary_residual_functions} yields an upper bound $d^{O(n^2)}$ on $\rho$ and thus proves a lower bound on the convergence rate $t^{-\frac{1}{2.5n\rho}}$ of the SOS relaxation in terms of $n$ and $d$ only. 

\begin{remark}
There are other applications of the Lojasiewicz inequality to polynomial optimization in the literature. For instance, it was shown in~\cite{KS15} that the {\L}ojasiewicz inequality~\eqref{ineq:Kurdyka_bound} can be used to reduce~\eqref{poly_optim_def} to minimization over a ball. 
\end{remark}

\section{Conclusion}
\label{sec:conclusion}
In this paper, we proved a nearly tight upper bound on the {\L}ojasiewicz exponent for semi-algebraic functions over a real closed field $\R$ in a very general setting. 
Unlike the previous best known bound in this setting due to Solern\'o 
\cite{S91}, our bound is independent of the cardinalities of the semi-algebraic descriptions of $f$, $g$, and $A$. 
We exploited this fact to improve the best known error bounds for polynomial and non-linear semi-definite systems.
As an abstraction of the notion of independence from the combinatorial parameters,
we proved a version of {\L}ojasiewicz inequality in polynomially bounded o-minimal structures. We proved  existence of a common {\L}ojasiewicz exponent for certain combinatorially defined infinite (but not necessarily definable) families of pairs of functions, which improves a prior result due to Chris Miller \cite{Miller}.

We end with a few open problems.
\vspace{5px}
\noindent
We proved in Theorem~\ref{thm:error_bound_semialg_set} that the exponent $\rho = d^{O(n)}$ in the error bound with respect to a zero-dimensional semi-algebraic set $M$, and Example~\ref{ex:exponential_dependence} indicates that this bound is indeed tight. 
Without the assumption on the dimension on $M$, the general bound on the exponent $\rho$ in Theorem~\ref{thm:error_bound_semialg_set} is $d^{O(n^2)}$.
There are some indications  in~\cite{K99} for generating examples whose {\L}ojasiewicz exponent is worse than Example~\ref{ex:exponential_dependence}. However, we have not been able to find any example with $\rho = d^{O(n^2)}$, so we do not know if this bound is tight as well. It would be interesting to resolve this gap.

Another interesting question is to prove an upper bound that depends on $\dim M$ which interpolates between the $0$-dimensional and the general case. More precisely, is it possible to improve the upper bound in
Theorem~\ref{thm:error_bound_semialg_set} to $d^{O(n \cdot \dim M)}$ ?

While the emphasis in the current paper has been on proving a bound on the {\L}ojasiewicz exponent which is independent
of the combinatorial parameter there is a special case that merits attention and in which the combinatorial parameter
may play a role. It is well known \cite{Bar97} that the topological complexity (say measured in terms of the Betti numbers) of a real algebraic set in $\R^n$ defined by $s$ quadratic equations is bounded by $n^{O(s)}$. This bound
(unlike the bounds discussed in the current paper) is polynomial in $n$ for fixed $s$. One could
ask if a similar bound also holds in this setting for the {\L}ojasiewicz exponent. 

\section*{Acknowledgements}
The first author was partially supported by NSF grants CCF-1910441 and CCF-2128702. The second author was partially supported by NSF grant CCF-2128702.

\bibliographystyle{siam}
\bibliography{mybibfile}

\def\cprime{$'$} \def\cprime{$'$}
\begin{thebibliography}{10}

\bibitem{BMo22}
{\sc L.~Baldi and B.~Mourrain}, {\em On the effective {P}utinar's
  positivstellensatz and moment approximation},  (2022).
\newblock arXiv:2111.11258v2 \url{https://arxiv.org/abs/2111.11258}.

\bibitem{Bar97}
{\sc A.~I. Barvinok}, {\em On the {B}etti numbers of semi-algebraic sets
  defined by few quadratic inequalities}, Math. Z., 225 (1997), pp.~231--244.

\bibitem{Basu9}
{\sc S.~Basu}, {\em Combinatorial complexity in o-minimal geometry}, Proc.
  London Math. Soc. (3), 100 (2010), pp.~405--428.

\bibitem{B17}
{\sc S.~Basu}, {\em Algorithms in real algebraic geometry: a survey}, in Real
  algebraic geometry, vol.~51 of Panor. Synth\`eses, Soc. Math. France, Paris,
  2017, pp.~107--153.

\bibitem{BM22}
{\sc S.~Basu and A.~Mohammad-Nezhad}, {\em On the central path of semi-definite
  optimization: Degree and worst-case convergence rate}, SIAM Journal on
  Applied Algebra and Geometry, 6 (2022), pp.~299--318.

\bibitem{BPR95}
{\sc S.~Basu, R.~Pollack, and M.-F. Roy}, {\em On the combinatorial and
  algebraic complexity of quantifier elimination}, J. ACM, 43 (1996),
  pp.~1002--1045.

\bibitem{BPRbook2posted}
\leavevmode\vrule height 2pt depth -1.6pt width 23pt, {\em Algorithms in real
  algebraic geometry}, vol.~10 of Algorithms and Computation in Mathematics,
  Springer-Verlag, Berlin, version bpr-ed2-posted3 from 27/06/2016 available at
  \url{http://perso.univ-rennes1.fr/marie-francoise.roy/bpr-ed2-posted3.html}.

\bibitem{BR10}
{\sc S.~Basu and M.-F. Roy}, {\em Bounding the radii of balls meeting every
  connected component of semi-algebraic sets}, Journal of Symbolic Computation,
  45 (2010), pp.~1270--1279.

\bibitem{BR2021}
{\sc S.~Basu and M.-F. Roy}, {\em Quantitative curve selection lemma},
  Mathematische Zeitschrift,  (2021).

\bibitem{BCR98}
{\sc J.~Bochnak, M.~Coste, and M.-F. Roy}, {\em Real Algebraic Geometry},
  Springer, 1998.

\bibitem{BLY14}
{\sc J.~M. Borwein, G.~Li, and L.~Yao}, {\em Analysis of the convergence rate
  for the cyclic projection algorithm applied to basic semi-algebraic convex
  sets}, SIAM Journal on Optimization, 24 (2014), pp.~498--527.

\bibitem{Br91}
{\sc L.~Br{\"o}cker}, {\em On basic semi-algebraic sets}, Expositiones
  Mathematicae, 9 (1991), pp.~289--334.

\bibitem{Michel2}
{\sc M.~Coste}, {\em An introduction to o-minimal geometry}, Istituti
  Editoriali e Poligrafici Internazionali, Pisa, 2000.
\newblock Dip. Mat. Univ. Pisa, Dottorato di Ricerca in Matematica.

\bibitem{DK05}
{\sc D.~D'Acunto and K.~Kurdyka}, {\em Explicit bounds for the {{\L}}ojasiewicz
  exponent in the gradient inequality for polynomials}, Annales Polonici
  Mathematici, 87 (2005), pp.~51--61.

\bibitem{Kl02}
{\sc E.~de~Klerk}, {\em Aspects of Semi-definite Programming: Interior Point
  Algorithms and Selected Applications}, vol.~65 of Series Applied
  Optimization, Springer, 2006.

\bibitem{Gabrielov-Vorobjov}
{\sc A.~Gabrielov and N.~Vorobjov}, {\em Approximation of definable sets by
  compact families, and upper bounds on homotopy and homology}, J. Lond. Math.
  Soc. (2), 80 (2009), pp.~35--54.

\bibitem{G99}
{\sc J.~Gwo{\'{z}}dziewicz}, {\em The {{\L}}ojasiewicz exponent of an analytic
  function at an isolated zero}, Commentarii Mathematici Helvetici, 74 (1999),
  pp.~364--375.

\bibitem{Hormander2}
{\sc L.~H\"{o}rmander}, {\em On the division of distributions by polynomials},
  Ark. Mat., 3 (1958), pp.~555--568.

\bibitem{H58}
{\sc L.~H{\"o}rmander}, {\em On the division of distributions by polynomials},
  Arkiv f{\"o}r Matematik, 3 (1958), pp.~555--568.

\bibitem{JKS92}
{\sc S.~Ji, J.~Kollar, and B.~Shiffman}, {\em A global {{\L}}ojasiewicz
  inequality for algebraic varieties}, Transactions of the American
  Mathematical Society, 329 (1992), pp.~813--818.

\bibitem{K99}
{\sc J.~Koll{\'a}r}, {\em An effective {{\L}}ojasiewicz inequality for real
  polynomials}, Periodica Mathematica Hungarica, 38 (1999), pp.~213--221.

\bibitem{KS14}
{\sc K.~Kurdyka and S.~Spodzieja}, {\em Separation of real algebraic sets and
  the Łojasiewicz exponent}, Proceedings of the American Mathematical
  Society., 142 (2014-9-14), pp.~3089--3102.

\bibitem{KS15}
\leavevmode\vrule height 2pt depth -1.6pt width 23pt, {\em Convexifying
  positive polynomials and sums of squares approximation}, SIAM Journal on
  Optimization, 25 (2015), pp.~2512--2536.

\bibitem{KS16}
{\sc K.~Kurdyka, S.~Spodzieja, and A.~Szlachci{\'{n}}ska}, {\em Metric
  properties of semi-algebraic mappings}, Discrete {\&} Computational Geometry,
  55 (2016), pp.~786--800.

\bibitem{KSS19}
\leavevmode\vrule height 2pt depth -1.6pt width 23pt, {\em Correction to:
  Metric properties of semi-algebraic mappings}, Discrete {\&} Computational
  Geometry, 62 (2019), pp.~990--991.

\bibitem{L01}
{\sc J.~B. Lasserre}, {\em Global optimization with polynomials and the problem
  of moments}, SIAM Journal on Optimization, 11 (2001), pp.~796--817.

\bibitem{L09}
{\sc M.~Laurent}, {\em Sums of Squares, Moment Matrices and Optimization Over
  Polynomials}, Springer New York, New York, NY, 2009, pp.~157--270.

\bibitem{LP98}
{\sc A.~S. Lewis and J.-S. Pang}, {\em Error Bounds for Convex Inequality
  Systems}, Springer US, Boston, MA, 1998, pp.~75--110.

\bibitem{Li10}
{\sc G.~Li}, {\em On the asymptotically well behaved functions and global error
  bound for convex polynomials}, SIAM Journal on Optimization, 20 (2010),
  pp.~1923--1943.

\bibitem{LMP15}
{\sc G.~Li, B.~S. Mordukhovich, and T.~S. Pham}, {\em New fractional error
  bounds for polynomial systems with applications to h{\"o}lderian stability in
  optimization and spectral theory of tensors}, Mathematical Programming, 153
  (2015), pp.~333--362.

\bibitem{Loj4}
{\sc S.~{\L}ojasiewicz}, {\em Sur le probl\`eme de la division}, Studia Math.,
  18 (1959), pp.~87--136.

\bibitem{Loj2}
{\sc S.~{\L}ojasiewicz}, {\em Triangulation of semi-analytic sets.}, Ann.
  Scuola Norm. Sup. Pisa, Sci. Fis. Mat., 18 (1964), pp.~449--474.

\bibitem{Loj1}
\leavevmode\vrule height 2pt depth -1.6pt width 23pt, {\em Ensembles
  semi-analytiques}.
\newblock preprint, IHES, 1965.

\bibitem{Loj3}
{\sc S.~{\L}ojasiewicz}, {\em Sur les ensembles semi-analytiques},  (1971),
  pp.~237--241.

\bibitem{LL94}
{\sc X.-D. Luo and Z.-Q. Luo}, {\em Extension of {H}offman's error bound to
  polynomial systems}, SIAM Journal on Optimization, 4 (1994), pp.~383--392.

\bibitem{L2000}
{\sc Z.-Q. Luo}, {\em New error bounds and their applications to convergence
  analysis of iterative algorithms}, Mathematical Programming, 88 (2000),
  pp.~341--355.

\bibitem{LP94}
{\sc Z.-Q. Luo and J.-S. Pang}, {\em Error bounds for analytic systems and
  their applications}, Mathematical Programming, 67 (1994), pp.~1--28.

\bibitem{Miller}
{\sc C.~Miller}, {\em Expansions of the real field with power functions}, Ann.
  Pure Appl. Logic, 68 (1994), pp.~79--94.

\bibitem{Milnor3}
{\sc J.~Milnor}, {\em Singular points of complex hypersurfaces}, Annals of
  Mathematics Studies, No. 61, Princeton University Press, Princeton, N.J.;
  University of Tokyo Press, Tokyo, 1968.

\bibitem{GP02}
{\sc L.~M.~G. na~Drummond and Y.~Peterzil}, {\em The central path in smooth
  convex semidefinite programs}, Optimization, 51 (2002), pp.~207--233.

\bibitem{OSS21}
{\sc B.~Osi{\'{n}}ska-Ulrych, G.~Skalski, and S.~Spodzieja}, {\em Effective
  {{\L}}ojasiewicz gradient inequality for {Nash} functions with application to
  finite determinacy of germs}, Journal of the Mathematical Society of Japan,
  73 (2021), pp.~277--299.

\bibitem{OSS21b}
{\sc B.~Osi{\'{n}}ska-Ulrych, G.~Skalski, and A.~Szlachci{\'{n}}ska}, {\em
  {{\L}}ojasiewicz inequality for a pair of semi-algebraic functions}, Bulletin
  des Sciences Mathématiques, 166 (2021), p.~102927.

\bibitem{Pang97}
{\sc J.-S. Pang}, {\em Error bounds in mathematical programming}, Mathematical
  Programming, 79 (1997), pp.~299--332.

\bibitem{P03}
{\sc P.~A. Parrilo}, {\em Semi-definite programming relaxations for
  semi-algebraic problems}, Math. Program., 96 (2003), pp.~293--320.
\newblock Algebraic and geometric methods in discrete optimization.

\bibitem{S91}
{\sc P.~Solern{\'o}}, {\em Effective {{\L}}ojasiewicz inequalities in
  semi-algebraic geometry}, Applicable Algebra in Engineering, Communication
  and Computing, 2 (1991), pp.~1--14.

\bibitem{Starchenko-Bourbaki}
{\sc S.~Starchenko}, {\em N{IP}, {K}eisler measures and combinatorics},
  no.~390, 2017, pp.~Exp. No. 1114, 303--334.
\newblock S\'{e}minaire Bourbaki. Vol. 2015/2016. Expos\'{e}s 1104--1119.

\bibitem{Dries}
{\sc L.~van~den Dries}, {\em Tame topology and o-minimal structures}, vol.~248
  of London Mathematical Society Lecture Note Series, Cambridge University
  Press, Cambridge, 1998.

\bibitem{W78}
{\sc R.~J. Walker}, {\em Algebraic Curves}, Springer, New York, NY, USA, 1978.

\end{thebibliography}
\end{document}